\documentclass[12pt]{article}
\usepackage{amsfonts}
\usepackage{amsmath}
\usepackage{amssymb}
\usepackage[hidelinks]{hyperref}
\usepackage{graphicx}
\usepackage[longnamesfirst]{natbib}
\setcitestyle{round}

\renewcommand{\mathbf}[1]{\boldsymbol{#1}}
\newtheorem{theorem}{Theorem}

\newtheorem{definition}{Definition}
\newtheorem{lemma}{Lemma}
\newenvironment{proof}[1][Proof]{\noindent\textbf{#1.} }{\ \rule{0.5em}{0.5em}}

\begin{document}

\title{Composite Robust Estimators for Linear Mixed Models}
\author{\textbf{C. Agostinelli}\\Department of Environmental Sciences, Informatics and Statistics \\Ca' Foscari University \\Venice, Italy \\E-mail: claudio@unive.it \\and \\\textbf{V.J. Yohai}\\Departamento de Matem\'aticas and Instituto de C\'alculo \\ FCEyN, University of Buenos Aires and CONICET \\ Buenos Aires, Argentina \\ E-mail: victoryohai@gmail.com}
\date{\today}
\maketitle

\begin{abstract}
The Classical Tukey-Huber Contamination Model (CCM) is a usual framework to
describe the mechanism of outliers generation in robust statistics. In a data
set with $n$ observations and $p$ variables, under the CCM, an outlier is a
unit, even if only one or few values are corrupted. Classical robust
procedures were designed to cope with this setting and the impact of
observations were limited whenever necessary. Recently, a different mechanism
of outliers generation, namely Independent Contamination Model (ICM), was
introduced. In this new setting each cell of the data matrix might be
corrupted or not with a probability independent on the status of the other
cells. ICM poses new challenge to robust statistics since the percentage of
contaminated rows dramatically increase with $p$, often reaching more than
50\% . When this situation appears, classical affine equivariant robust
procedures do not work since their breakdown point is 50\%. For this
contamination model we propose a new type of robust methods namely composite
robust procedures which are inspired on the idea of composite likelihood,
where low dimension likelihood, very often the likelihood of pairs, are
aggregate together in order to obtain an approximation of the full likelihood
which is more tractable. Our composite robust procedures are build over pairs
of observations in order to gain robustness in the independent contamination
model. We propose composite S and $\tau$-estimators for linear mixed models.
Composite $\tau$-estimators are proved to have an high breakdown point both in
the CCM and ICM. A Monte Carlo study shows that our estimators compare
favorably with respect to classical S-estimators under the CCM and outperform
them under the ICM. One example based on a real data set illustrates the new
robust procedure.


\end{abstract}

\section{Introduction}

\label{sec:introduction} The purpose of this paper is to find robust
procedures for mixed linear models. This class of models include among others
ANOVA models with repeated measures, models with random nested design and
models for studying longitudinal data. These models are generally based on the
assumption that the data follow a normal distribution and therefore the
parameters are estimated using the maximum likelihood principle. See for
example, \citet{searle_casella_mcculloch_1992}. As is well known, in general,
the estimator obtained by maximum likelihood under the assumption that the
data have a normal distribution is very sensitive to the presence of a small
fraction of outliers in the sample. More than that, just one outlier may have
an unbounded effect on this estimator. There are many robust estimators that
have been proposed to avoid a large outlier influence. A large list of
references of these proposals is available in \citet{heritier2009}.
\citet{victoria-feser_high_2006} introduces a very interesting robust
S-estimator for mixed linear models based on M-scales which has breakdown
point equal to $0.5$. We can also mention \cite{gill_2000},
\citet{jiang_zhang_2001}, \citet{sinha_2004}, \citet{copt_heritier_2006},
\citet{jacqmin-gadda_sibillot_proust_molina_thiebaut_2007},
\citet{lachosa_deyb_canchoc_2009}, \citet{chervonevavishnyakov2011} and
\citet{koller_2013a} which studied an SMDM-estimator. The procedure proposed
in the last paper is implemented in the \texttt{R} package \texttt{robustlmm} \citep{koller_2013b}.

However all these procedures are focused on coping with outliers generated
under the Classical (Tukey-Huber) Contamination Model (CCM), where some
percentage of the units that compose the sample are replaced by outliers.
However \citet{alqallaf_propagation_2009} introduced another type of
contamination (called Independent Contamination Model, ICM) that may occur in
multivariate data. Instead of contaminating a percentage of the units that
compose the sample, the different cells of each unit may be independently
contaminated. In this case, if the dimension of each unit is large, even a
small fraction of cell contamination may lead to a large fraction of units
with at least one contaminated cell. This type of contamination specially
occurs when the different variables that compose each unit are measured from
independent laboratories. \citet{alqallaf_propagation_2009} showed that for
this type of contamination the breakdown point of affine equivariant
procedures for multivariate location and covariance matrix tends to zero when
the number of variables increases and therefore their degree of robustness is
not satisfactory. A similar phenomenon occurs when dealing with mixed linear
models. In particular the S-estimator procedure introduced in
\citet{victoria-feser_high_2006} loses robustness for high dimensional data
with independent contamination.

In this paper we propose a new class of robust estimators for linear mixed
models. These estimators may be thought as robust counterparts of the
composite likelihood estimators proposed by \citet{lindsay_composite_1988}. If
a vector $\mathbf{y}$ of dimension $p$ is observed, the composite likelihood
estimators are based on the likelihood of all the subvectors of a dimension
$p^{\ast}<p$. The estimators that we propose here are based on $\tau$-scales
of the Mahalanobis distances of two dimensional subvectors of $\mathbf{y}$.
The $\tau$-scale estimators were introduced by \citet{yohai_zamar_1988} and
provides scales estimators which are simultaneously highly robust and highly
efficient. We are going to show that these estimators have a robust behavior
for both contamination models: the classical contamination model and the
independent contamination model. In particular we will show that the breakdown
point for the classical contamination model is $0.5$, while for the
independent contamination model is $0.25$.

In Section \ref{sec:modelnotation} the model and the notation are presented.
Section \ref{sec:compositeS} introduces the Composite S-estimator, while
Section \ref{sec:compositeT} defines the Composite $\tau$-estimator. Sections
\ref{sec:breakdown} and \ref{sec:asymptotic} discuss the breakdown properties
and the asymptotic normality of the Composite $\tau$-estimator, Section
\ref{sec:computation} provides details on the computational algorithm and
Section \ref{sec:examples} illustrates with a real data set the advantages of
the proposed estimator. In Section \ref{sec:simulations} we perform a Monte
Carlo simulation that shows that the proposed procedure has a robust behavior
under both contamination models. Section \ref{sec:conclusions} provides some
concluding remarks. An Appendix contains details on computational aspects and
the proofs of statements reported in previous Sections.

\section{Model and Notation \label{sec:modelnotation}}

Denote by $N_{p}(\mathbf{\mu},\Sigma)$ the multivariate normal distribution of
dimension $p$ with mean $\mathbf{\mu}$ and covariance matrix $\Sigma$. Many
statistical models for components of variance and longitudinal analysis are of
the following form. In the case of fixed covariables is assumed that $n$
independent $p$-dimensional $\operatorname{rand}$om vectors $\mathbf{y}%
_{1},\ldots,\mathbf{y}_{n}$ in $\mathbb{R}^{p}$ are observed, and
$\mathbf{y}_{i},1\leq i\leq n$ has distribution $N_{p}(\mathbf{\mu}%
_{i}(\mathbf{\beta}),\Sigma(\eta,\mathbf{\gamma}))$, where
\begin{align}
\mathbf{\mu}_{i}(\mathbf{\beta})  &  =(\mu_{i1}(\mathbf{\beta}),\ldots
,\mu_{ip}(\mathbf{\beta}))^{\top}\label{mu}\\
&  =\mathbf{x}_{i}\mathbf{\beta},\text{ }1\leq i\leq n,\nonumber
\end{align}
$\mathbf{x}_{1},\ldots,\mathbf{x}_{n}$ are a fixed $p\times k$ matrices and
$\mathbf{\beta}\in\mathbb{R}^{k}$ is an unknown $k$-vector parameter.
Moreover,
\begin{equation}
\Sigma(\eta,\mathbf{\gamma})=\eta(V_{0}+\sum_{j=1}^{J}\gamma_{j}V_{j}),
\label{cov}%
\end{equation}
where $V_{j}$, $1\leq j\leq J$ are $p\times p$ matrices, $V_{0}$ is the
$p\times p$ identity, $\eta>0$ and $\mathbf{\gamma}=(\gamma_{1},\ldots
,\gamma_{J})^{\top}\in\Gamma$ are unknown parameters, where
\begin{equation*}
\Gamma=\{\mathbf{\gamma}\in\mathbb{R}^{J}:\ \Sigma(1,\mathbf{\gamma}%
)\quad\text{is positive definite}\}.
\end{equation*}
In the case of random covariables, that is, when $\mathbf{x}_{1}%
,\ldots,\mathbf{x}_{n}$ are i.i.d random matrices, it is assumed that%
\begin{equation}
\mathbf{y}_{i}|\mathbf{x}_{i}\sim N_{p}\left(  \mathbf{\mu}_{i}(\mathbf{\beta
}),\Sigma(\eta,\mathbf{\gamma})\right)  . \label{cond}%
\end{equation}
This is equivalent to $\mathbf{u}_{i}=\mathbf{y}_{i}-\mathbf{\mu}%
_{i}(\mathbf{\beta})$ independent of $\mathbf{x}_{i}$ with distribution
$N_{p}(\mathbf{0},\Sigma(\eta,\mathbf{\gamma}))$. However, in Section
\ref{sec:asymptotic}, where we study the asymptotic properties of the proposed
estimators, we use a weaker assumption. In fact we only require that
$\mathbf{u}_{i}$ be independent of $\mathbf{x}_{i}$ and have elliptical
distribution with center $\mathbf{0}$ and covariance matrix $\Sigma
(\eta,\mathbf{\gamma})$.

This setup covers several statistical models, for instance that of the form
\begin{equation}
\mathbf{y}_{i}=\mathbf{x}_{i}\mathbf{\beta}+\sum_{j=1}^{J}\mathbf{z}_{j}%
\zeta_{ij}+\mathbf{\varepsilon}_{i},\text{ }1\leq i\leq n,
\label{modelexample}%
\end{equation}
where $\mathbf{x}_{i}$ are as before, $\mathbf{z}_{j}$, $1\leq j\leq J$, are
$p\times q_{j}$ known design matrices for the random effects, $\zeta_{ij}$ are
independent $q_{j}$-dimensional vectors with distribution $N_{q_{j}}%
(0,\sigma_{j}^{2}\mathbf{I}_{q_{j}})$, where\textbf{ }$\mathbf{I}_{p}$ is the
$p\times p$ identity and $\mathbf{\varepsilon}_{i}$ $(1\leq i\leq n)$ are
$p$-dimensional error vectors with distribution $N(0,\sigma_{0}^{2}I_{p})$.
Then, in this case we have $\eta=\sigma_{0}^{2},\mathbf{\gamma}=(\gamma
_{1},\ldots,\gamma_{J})^{\top}$ with $\gamma_{j}=\sigma_{j}^{2}/\sigma_{0}%
^{2}>0$, $V_{j}=\mathbf{z}_{j}\mathbf{z}_{j}^{\top},$ $1\leq j\leq J$.

\subsection{ Composite S-estimator \label{sec:compositeS}}

A very interesting class of S-estimators for the model defined by (\ref{mu})
and (\ref{cov}) was proposed by \citet{victoria-feser_high_2006}.

Given a $p$ dimensional column vector $\mathbf{y}$ and a vector $\mathbf{\mu}$
and $p\times p$ matrix $\Sigma$ the square of the Mahalanobis distance is
defined by
\begin{equation*}
m(\mathbf{y},\mathbf{\mu},\Sigma)=(\mathbf{y}-\mathbf{\mu})^{\top}\Sigma
^{-1}(\mathbf{y}-\mathbf{\mu}).
\end{equation*}
Let $\rho:\mathbb{R}^{+}\rightarrow\mathbb{R}^{+}$, where $\mathbb{R}^{+}$ is
the set of nonnegative real numbers, satisfying the following properties:

\begin{description}
\item[A1] $\rho(0)=0$.

\item[A2] $0\leq v\leq v^{\ast}$ implies $\rho(v)\leq\rho(v^{\ast})$.

\item[A3] $\rho$ is continuous.

\item[A4] sup$_{v}\rho(v)=1$.

\item[A5] If $\rho(u)<1$ and $0\leq u<v$, then $\rho(u)<\rho(v)$.
\end{description}

Let $b$ be defined by
\begin{equation*}
E_{\chi_{_{p}}^{2}}(\rho(v))=b,
\end{equation*}
where $v\sim\chi_{p}^{2}$ is a chi-square distribution with $p$ degrees of
freedom. Then, given a sample $\mathbf{m}=(m_{1},\ldots,m_{n})^{\top}$, an
M-scale estimator $s(\mathbf{m})$ is defined by the value $s$ solution of
\begin{equation}
\frac{1}{n}\sum_{i=1}^{n}\rho\left(  \frac{m_{i}}{s}\right)  =b. \label{Msc}%
\end{equation}
The S-estimator proposed by \citet{victoria-feser_high_2006} is defined by
\begin{equation*}
(\widehat{\mathbf{\beta}},\widehat{\eta},\widehat{\mathbf{\gamma}})=\arg
\min\det_{\eta,\mathbf{\gamma}}\Sigma(\eta,\mathbf{\gamma})
\end{equation*}
subject to
\begin{equation*}
s(m(\mathbf{y}_{1},\mathbf{\mu}_{1}(\mathbf{\beta}),\Sigma(\eta,\mathbf{\gamma
})),\ldots,m(\mathbf{y}_{n},\mathbf{\mu}_{n}(\mathbf{\beta}),\Sigma
(\eta,\mathbf{\gamma})))=1.
\end{equation*}
These estimators can be thought as a constrained version of the S-estimators
for multidimensional location and scatter proposed by \citet{davies_asymptotic_1987}.

Given a squared matrix $A$ we denote by $A^{\ast}=A/|A|^{1/p}$ where $|A|$ is
the determinant of the matrix $A$. Note that $\Sigma^{\ast}(\eta
,\mathbf{\gamma})$ depends only on $\mathbf{\gamma}$ and then will be denoted
by $\Sigma^{\ast}(\mathbf{\gamma})$. It is easy to show that the estimators
proposed by \citet{victoria-feser_high_2006} can be also defined by
\begin{align*}
(\widehat{\mathbf{\beta}},\widehat{\mathbf{\gamma}})  &  =\arg\min
_{\mathbf{\beta,\gamma}}s(m(\mathbf{y}_{1},\mathbf{\mu}_{1}(\mathbf{\beta
}),\Sigma^{\ast}(\mathbf{\gamma})),\ldots,m(\mathbf{y}_{n},\mathbf{\mu}%
_{n}(\mathbf{\beta}),\Sigma^{\ast}(\mathbf{\gamma}))),\\
\widehat{\eta}  &  =s(m(\mathbf{y}_{1},\mathbf{\mu}_{1}(\widehat
{\mathbf{\beta}}),\Sigma(1,\widehat{\mathbf{\gamma}})),\ldots,m(\mathbf{y}%
_{n},\mathbf{\mu}_{n}(\widehat{\mathbf{\beta}}),\Sigma(1,\widehat
{\mathbf{\gamma}}))),
\end{align*}
where the M-scale $s$ is defined now by (\ref{Msc}). Notice that $\widehat{\eta}$
is defined by
\begin{equation*}
\frac{1}{n}\sum_{i=1}^{n}\rho\left(  \frac{m(\mathbf{y}_{i},\mu_{i}%
(\widehat{\mathbf{\beta}}),\Sigma(1,\widehat{\mathbf{\gamma}}))}{\widehat
{\eta}}\right)  =b.
\end{equation*}
In the classical contamination model a fraction $\varepsilon$ of the vectors
$\mathbf{y}_{i}$ are replaced by outliers. \citet{victoria-feser_high_2006}
show that for this model the breakdown point of this estimator is
$\varepsilon^{\ast}=\min(b,1-b)$. Therefore if $b=0.5$, we get $\varepsilon
^{\ast}=0.5$.

\citet{alqallaf_propagation_2009} consider a different contamination model for
multivariate data: the independent contamination model. In this contamination
model if we observe a vector $\mathbf{y}_{i}=(y_{i1},\ldots,y_{ip})^{\top}$
each component $y_{ij}$ of $\mathbf{y}_{i}$ has probability $\varepsilon$ of
being replaced by an outlier. Therefore the probability that at least one
component of $\mathbf{y}_{i}$ be contaminated is $1-(1-\varepsilon)^{p}$, and
this number is close to one when $p$ is large even if $\varepsilon$ is small.

\citet{alqallaf_propagation_2009} showed that the breakdown point for the
independent contamination model of S-estimators of multivariate location and
scatter tends to $0$ when $p\rightarrow\infty$. The same happens with other
popular affine equivariance estimators like the minimum volume ellipsoid
\citep{rousseeuw_multivariate_1985}, Minimum covariance determinant
\citep{rousseeuw_multivariate_1985} or the Donoho-Stahel estimators
\citep{donoho_breakdown_1982, stahel_robuste_1981}. The S-estimator proposed
by \citet{victoria-feser_high_2006} for model (\ref{mu})-(\ref{cov}) have a
similar shortcoming: when $p\rightarrow\infty$, its breakdown point tends to
$0$ under the independent contamination model. For this reason, hereafter we
introduce a new type of estimators namely composite S-estimators and composite
$\tau$-estimators.

Given a vector $\mathbf{a}=(a_{1},\ldots,a_{p})^{\top}$, a $p\times p$ matrix
$A$ and for a couple $(j,l)$ of indices ($1\leq j<l\leq p$) we denote
$\mathbf{a}^{jl}=(a_{j},a_{l})^{\top}$ and $A_{jl}$ the submatrix
\begin{equation*}
A_{jl}=\left(
\begin{array}
[c]{ll}%
a_{jj} & a_{jl}\\
a_{lj} & a_{ll}%
\end{array}
\right)  .
\end{equation*}
In a similar way, given a $p\times k$ matrix $\mathbf{x}$ we denote by
$\mathbf{x}^{jl}$ the matrix of dimension $2\times k$ built by using the
corresponding $(j,l)$ rows of $\mathbf{x}$. We define a pairwise squared
Mahalanobis distance and a pairwise scale by
\begin{align}
m_{i}^{jl}(\mathbf{\beta},\mathbf{\gamma})  &  =m(\mathbf{y}_{i}%
^{jl},\mathbf{\mu}_{i}^{jl}(\mathbf{\beta}),\Sigma_{jl}^{\ast}(\mathbf{\gamma
})),\nonumber\\
s_{jl}(\mathbf{\beta},\mathbf{\gamma})  &  =s(m_{1}^{jl}(\mathbf{\beta
},\mathbf{\gamma}),\ldots,m_{n}^{jl}(\mathbf{\beta},\mathbf{\gamma})),
\label{sij}%
\end{align}
where the M-scale $s$ is now defined by (\ref{Msc}) with $b$ given by
\begin{equation}
E_{\chi_{2}^{2}}(\rho(v))=b. \label{Msc2}%
\end{equation}
Thus $S(\mathbf{\beta},\mathbf{\gamma})$ is defined by%
\begin{equation}
S(\mathbf{\beta},\mathbf{\gamma})=\sum_{j=1}^{p-1}\sum_{l=j+1}^{p}%
s_{jl}(\mathbf{\beta},\mathbf{\gamma}). \label{goal}%
\end{equation}
Similarly to \citet{victoria-feser_high_2006}, we define for the model in
(\ref{mu})-(\ref{cov}), the composite S-estimator of $\mathbf{\beta}$ and
$\mathbf{\gamma}$ by
\begin{equation}
(\widehat{\mathbf{\beta}},\widehat{\mathbf{\gamma}})=\arg\min_{\mathbf{\beta
,\gamma}}S(\mathbf{\beta},\mathbf{\gamma}) \label{equ:S}%
\end{equation}
and the estimator $\widehat{\eta}$ of $\eta$ by
\begin{equation}
\frac{2}{p(p-1)n}\sum_{i=1}^{n}\sum_{j=1}^{p-1}\sum_{l=j+1}^{p}\rho\left(
\frac{(\mathbf{y}_{i}^{jl}-\mathbf{x}_{i}^{jl}\widehat{\mathbf{\beta}})^{\top
}\Sigma_{jl}(1,\widehat{\mathbf{\gamma}})^{-1}(\mathbf{y}_{i}^{jl}%
-\mathbf{x}_{i}^{jl}\widehat{\mathbf{\beta}})}{\widehat{\eta}}\right)  =b.
\label{equ:eta0}%
\end{equation}
One shortcoming of the composite S-estimators are, as occurs with regression
S-estimators, that they are not simultaneously highly robust and highly
efficient. For this reason in next section we introduce the composite $\tau
$-estimators which are defined similarly to the $S$-estimators, but replacing
the M-scale by a $\tau$-scale.

\section{Composite $\tau$-Estimator \label{sec:compositeT}}

In this section we introduce the composite $\tau$-estimator. A $\tau$-scale is
defined using two functions $\rho_{1}$ and $\rho_{2}$. Given a sample
$\mathbf{m}=(m_{1},\ldots,m_{n})^{\top}$, the function $\rho_{1}$ is used to
define an M-scale $s$ by
\begin{equation}
\frac{1}{n}\sum_{i=1}^{n}\rho_{1}\left(  \frac{m_{i}}{s}\right)  =b,
\end{equation}
and the $\tau$-scale by
\begin{equation*}
\tau=s\frac{1}{n}\sum_{i=1}^{n}\rho_{2}\left(  \frac{m_{i}}{s}\right)  .
\end{equation*}
We will require that $\rho_{1}$ and $\rho_{2}$ satisfy A1-A5. Put $\psi
_{i}(v)=\rho_{i}^{\prime}(v)$, $i=1,2$. In \citet{yohai_zamar_1988} it is
shown that to guarantee the Fisher consistency of the $\tau$-estimators of
regression, it is required that $\rho_{2}$ satisfies the following condition:

\begin{description}
\item[A6] $\rho_{2}$ is continuously differentiable and $2\rho_{2}(v)-\psi
_{2}(v)v>0$ for $v>0$.
\end{description}

The breakdown point of the $\tau$-scale is the same as that of the $s$-scale.
Then we are going to set $b=0.5$ to have breakdown point close to $0.5$ in the
classical contamination model.

The estimators are going to be defined as in the previous section by replacing
the M-scales by the $\tau$-scales. Then $s_{jl}(\mathbf{\beta},\mathbf{\gamma
})$ is defined as in (\ref{sij}) and the $\tau$-scale is
\begin{equation*}
\tau_{jl}(\mathbf{\beta},\mathbf{\gamma})=s_{jl}(\mathbf{\beta},\mathbf{\gamma
})\frac{1}{n}\sum\rho_{2}\left(  \frac{m_{i}^{jl}(\mathbf{\beta}%
,\mathbf{\gamma})}{s_{jl}(\mathbf{\beta},\mathbf{\gamma})}\right)  .
\end{equation*}
Let $T(\mathbf{\beta},\mathbf{\gamma})$ be the sum of all the $\tau_{jl}$
scales, i.e.,
\begin{equation}
T(\mathbf{\beta},\mathbf{\gamma})=\sum_{j=1}^{p-1}\sum_{l=j+1}^{p}\tau
_{jl}(\mathbf{\beta},\mathbf{\gamma}), \label{goal-tau}%
\end{equation}
then the composite $\tau$-estimator of $(\mathbf{\beta,\gamma})$ is defined as
follows
\begin{equation}
(\widehat{\mathbf{\beta}},\widehat{\mathbf{\gamma}})=\arg\min_{\mathbf{\beta
},\mathbf{\gamma}}T(\mathbf{\beta},\mathbf{\gamma}), \label{deftaues}%
\end{equation}
while $\widehat{\eta}$ is obtained as in (\ref{equ:eta0}) setting $\rho
=\rho_{1}$.

In the example of Section \ref{sec:examples} and in the Monte Carlo study of
Section \ref{sec:simulations} we took $\rho_{i},$ $i=1,2$ in the following
family of functions,%

\begin{equation}
\rho_{c}^{\text{o}\ast}(u)=\rho_{c}^{\text{o}}(u^{1/2}),\label{optimalsq}%
\end{equation}
where $\rho_{c}^{\text{o}}$ is the family of rho functions introduced by
\citet{muler_yohai_2002} defined by
\begin{equation}
\rho_{c}^{\text{o}}(v)=\left\{
\begin{array}
[c]{lr}%
\frac{v^{2}}{2ac^{2}} & v\leq2\\
\frac{1}{a}\left(  \frac{a_{4}}{8}\frac{v^{8}}{c^{8}}+\frac{a_{3}}{6}%
\frac{v^{6}}{c^{6}}+\frac{a_{2}}{4}\frac{v^{4}}{c^{4}}+\frac{a_{1}}{2}%
\frac{v^{2}}{c^{2}}+a_{0}\right)   & 2<v\leq3\\
1 & v>3
\end{array}
\right.  \label{optimal}%
\end{equation}
where $a_{0}=1.792$, $a_{1}=-1.944$, $a_{2}=1.728$, $a_{3}=-0.312$,
$a_{4}=0.016$ and $a=3.250$. The functions in this family have shapes close to
that of those in the optimal family obtained by \citet{yohai_zamar_1997}.
However, they are easier to compute. The reason why we compose the function
$\rho_{c}^{\text{o}}(v)$ with the square root is that the functions
$\rho_{c}^{\text{o}\ast}$ are applied to the squared Mahalanobis distances.
Notice that for any $\lambda>0,$ the $\tau$ scale obtained with $\rho_{1}%
=\rho_{\lambda c_{1}}^{\text{o}\ast}$ and $\rho_{2}=\rho_{\lambda c_{2}%
}^{\text{o}\ast}$ is equal to the  $\tau$ scale corresponding to $\rho
_{1}=\rho_{c_{1}}^{\text{o}\ast}$ and $\rho_{2}=\rho_{c_{2}}^{\text{o}\ast}$
divided by $\lambda.$ Hence without loss of generality we can take
$\rho_{1}=\rho_{1}^{\text{o}\ast}$. We found that taking $\rho_{2}=\rho
_{1.64}^{\text{o}\ast}$ we obtain a good trade-off between robustness and
efficiency, and these are  the values  that we recommend to use.

It is easy to show that the composite $\tau$-estimators are equivariant for
regression transformations of the form $\mathbf{y}_{i}^{\ast}=\mathbf{y}%
_{i}+\mathbf{x}_{i}\mathbf{\delta}$ where $\mathbf{\delta}$ is a $k\times1$
vector, affine transformations of the form $\mathbf{x}_{i}^{\ast}%
=\mathbf{x}_{i}B$, where $B$ is a $k\times k$ non singular matrix or scale
transformations of the form $\mathbf{y}_{i}^{\ast}=\zeta y_{i}$, where $\zeta$
is a scalar.

\section{Breakdown point $_{{}}$\label{sec:breakdown}}

\citet{donoho_huber_1983} introduced the concept of a finite sample breakdown
point (FSBDP). For our case, let $\widehat{\mathbf{\beta}}$ and $\widehat
{\mathbf{\upsilon}}=\widehat{\eta}(1,\widehat{\mathbf{\gamma}})$ be estimators
of $\mathbf{\beta}$ and $\mathbf{\upsilon}=(\eta,\eta\mathbf{\gamma})$.
Informally speaking, the FSBDP of $\widehat{\mathbf{\beta}}$ is the smallest
fraction of outliers that makes the estimator unbounded.

To formalize this, let $\mathbf{T}$ be a data set of size $n$ corresponding to
model (\ref{mu})-(\ref{cov}), $\mathbf{T}=(\mathbf{t}_{1},\ldots
,\mathbf{t}_{n})$, $\mathbf{t}_{i}=(\mathbf{y}_{i},\mathbf{x}_{i}%
)=(\mathbf{t}_{i1},\ldots,\mathbf{t}_{ip})^{\top}$, $\mathbf{y}_{i}%
\in\mathbb{R}^{p}$, $\mathbf{x}_{i}\in\mathbb{R}^{p\times k}$ and
$\mathbf{t}_{ij}=(y_{ij},x_{ij1},\ldots,x_{ijk})$ ($1\leq j\leq p$). Let
$\mathcal{T}_{m}^{(\text{C})}$ be the set of all the samples $\mathbf{\check
{T}}=(\check{\mathbf{t}}_{1},\ldots,\check{\mathbf{t}}_{n})$ with
$\check{\mathbf{t}}_{i}=(\check{\mathbf{t}}_{i1},\ldots,\check{\mathbf{t}%
}_{ip})^{\top}$ such that $\#\{i:\check{\mathbf{t}}_{i}=\mathbf{t}_{i}\}\geq
n-m$. Given estimators $\widehat{\mathbf{\beta}}$ and $\widehat
{\mathbf{\upsilon}}$ we let
\begin{align*}
B_{m}^{(\text{C})}(\mathbf{T},\widehat{\mathbf{\beta}})  &  =\sup
\{\Vert\widehat{\mathbf{\beta}}(\mathbf{\check{T}})\Vert,\mathbf{\check{T}}%
\in\mathcal{T}_{m}^{(\text{C})}\},\\
B_{m}^{+(\text{C})}(\mathbf{T},\widehat{\mathbf{\upsilon}})  &  =\sup
\{\Vert\widehat{\mathbf{\upsilon}}(\mathbf{\check{T}})\Vert,\mathbf{\check{T}%
}\in\mathcal{T}_{m}^{(\text{C})}\},\\
B_{m}^{-(\text{C})}(\mathbf{T},\widehat{\mathbf{\upsilon}})  &  =\inf
\{\Vert\widehat{\mathbf{\upsilon}}(\mathbf{\check{T}})\Vert,\mathbf{\check{T}%
}\in\mathcal{T}_{m}^{(\text{C})}\}.
\end{align*}

\begin{definition}
The finite sample breakdown point of $\widehat{\mathbf{\beta}}$ for classical
contamination (FSBDPCC) at the sample $\mathbf{T}$ is defined by
$\varepsilon^{(\text{C})}(\mathbf{T},\widehat{\mathbf{\beta}})=m^{\ast}/n$
where $m^{\ast}=\min\{m:B_{m}^{(\text{C})}(\mathbf{T},\widehat{\mathbf{\beta}%
})=\infty\}$ and the breakdown point of $\widehat{\mathbf{\upsilon}}$ by
$\varepsilon^{(\text{C})}(\mathbf{T},\widehat{\mathbf{\upsilon}})=m^{\ast}/n$
where
\begin{equation*}
m^{\ast}=\min\{m:\frac{1}{B_{m}^{-(\text{C})}(\mathbf{T},\widehat
{\mathbf{\upsilon}})}+B_{m}^{+(\text{C})}(\mathbf{T},\widehat{\mathbf{\upsilon
}})=\infty\}.
\end{equation*}

\end{definition}

Let $\mathcal{T}_{m}^{(\text{I}).}$ be the set of all the samples
$\mathbf{\check{T}}=(\check{\mathbf{t}}_{1},\ldots,\check{\mathbf{t}}_{n})$
such that $\#\{i:\check{\mathbf{t}}_{ij}=\mathbf{t}_{ij}\}\geq n-m$ for each
$j$, $1\leq j\leq p$. Given estimators $\widehat{\mathbf{\beta}}$ and
$\widehat{\mathbf{\upsilon}}$ we let
\begin{align*}
B_{m}^{(\text{I}).}(\mathbf{T},\widehat{\mathbf{\beta}})  &  =\sup
\{\Vert\widehat{\mathbf{\beta}}(\mathbf{\check{T}})\Vert,\mathbf{\check{T}}%
\in\mathcal{T}_{m}^{(\text{I})}\},\\
B_{m}^{+(\text{I})}(\mathbf{T},\widehat{\mathbf{\upsilon}})  &  =\sup
\{\Vert\widehat{\mathbf{\upsilon}}(\mathbf{\check{T}})\Vert,\mathbf{\check{T}%
}\in\mathcal{T}_{m}^{(\text{I})}\},\\
B_{m}^{-(\text{I})}(\mathbf{T},\widehat{\mathbf{\upsilon}})  &  =\inf
\{\Vert\widehat{\mathbf{\upsilon}}(\mathbf{\check{T}})\Vert,\mathbf{\check{T}%
}\in\mathcal{T}_{m}^{(\text{I})}\}.
\end{align*}

\begin{definition}
The finite sample breakdown point for $\widehat{\mathbf{\beta}}$ under
independent contamination (FSBDPIC) at the sample $\mathbf{T}$ is defined by
$\varepsilon^{(\text{I})}(\mathbf{T}$,$\widehat{\mathbf{\beta}})=m^{\ast}/n$
where $m^{\ast}=\min\{m:B_{m}^{(\text{I})}(\mathbf{T}$,$\widehat
{\mathbf{\beta}})=\infty\}$ and the breakdown point of $\widehat
{\mathbf{\upsilon}}$ by $\varepsilon^{(\text{I})}(\mathbf{T},\widehat
{\mathbf{\upsilon}})=m^{\ast}/n$ where
\begin{equation*}
m^{\ast}=\min\{m:\frac{1}{B_{m}^{-(\text{I})}(\mathbf{T},\widehat
{\mathbf{\upsilon}})}+B_{m}^{+(\text{I})}(\mathbf{T},\widehat{\mathbf{\upsilon
}})=\infty\}.
\end{equation*}

\end{definition}

The following theorems, whose proofs are discussed in Appendix
\ref{app:breakdown1}, gives a lower bound for the breakdown point of composite
$\tau$-estimators under both the classical and the independent contamination
models. Before to state the Theorems we need the following notation. Given a
sample $\mathbf{T}$ $=(\mathbf{t}_{1},\ldots,\mathbf{t}_{n})$ we define
\begin{align}
h_{jl}(\mathbf{T})  &  =\max_{\Vert\mathbf{b}\Vert>0}\#\{i:\mathbf{x}_{i}%
^{jl}\mathbf{b}=\mathbf{0}\},\label{equ:hjl}\\
h(\mathbf{T})  &  =\max_{jl}h_{jl}(\mathbf{T}),\label{equ:h}\\
h_{jl}^{\ast}(\mathbf{T})  &  =\max_{\Vert\mathbf{u}\Vert>0,\mathbf{b}%
}\#\{i:\mathbf{u}^{\top}(\mathbf{y}_{i}^{jl}-\mathbf{x}_{i}^{jl}%
\mathbf{b})=\mathbf{0}\},\label{equ:hjlstar}\\
h^{\ast}(\mathbf{T})  &  =\max_{jl}h_{jl}^{\ast}(\mathbf{T}),
\label{equ:hstar}\\
f(\mathbf{T})  &  =h(\mathbf{T})+h^{\ast}(\mathbf{T}). \label{equ:f}%
\end{align}

\begin{theorem}
\label{teo:classicbp} Let $\mathbf{T}=(\mathbf{t}_{1},\ldots,\mathbf{t}_{n})$,
$\mathbf{t}_{i}=(\mathbf{y}_{i},\mathbf{x}_{i})$, $f$ as defined in
(\ref{equ:f}). Assume that A1-A6 holds and let $(\widehat{\mathbf{\beta}%
},\widehat{\mathbf{\upsilon}})$ be the composite $\tau$-estimator for the
model given by (\ref{mu}) and (\ref{cov}). Then a lower bound for
$\varepsilon^{(\text{C})}(\mathbf{T},\widehat{\mathbf{\beta}})$ and for
$\varepsilon^{(\text{C})}(\mathbf{T},\widehat{\mathbf{\upsilon}})$ is given by
$\min((1-b)-f(\mathbf{T})/n,b)$.
\end{theorem}

Note that taking $b=0.5$, this lower bound is close to $0.5$ for large $n$
independently of $p$.

\begin{theorem}
\label{teo:independentbp} Let $\mathbf{T}=(\mathbf{t}_{1},\ldots
,\mathbf{t}_{n})$, $\mathbf{t}_{i}=(\mathbf{y}_{i},\mathbf{x}_{i})$, $f$ as
defined in (\ref{equ:f}). Assume that A1-A6 holds and let $(\widehat
{\mathbf{\beta}},\widehat{\mathbf{\upsilon}})$ be the composite $\tau
$-estimator for the model given by (\ref{mu}) and (\ref{cov}). Then a lower
bound for $\varepsilon^{(\text{I})}(\mathbf{T},\widehat{\mathbf{\beta}})$ and
for $\varepsilon^{(\text{I})}(\mathbf{T},\widehat{\mathbf{\upsilon}})$ is
given by $0.5\min((1-b)-f(\mathbf{T})/n,b)$.
\end{theorem}

In this case taking $b=0.5$ this lower bound is close to $0.25$ for large $n$
independently of $p$.

\section{Consistency and Asymptotic Normality \label{sec:asymptotic}}

In this Section we study the almost sure consistency and asymptotic normality
of the composite $\tau$-estimators. We need the following additional
assumptions for consistency

\begin{description}
\item[A7] The vector $\mathbf{x}$ is random and the error $\mathbf{u}%
=\mathbf{y}-\mathbf{x} \mathbf{\beta}_{0}$ is independent of $\mathbf{x}$ and
$\mathbf{u}$ has an elliptical density of the form
\begin{equation}
f(\mathbf{u})=\frac{f_{0}^{\ast}(\mathbf{u}^{\top}\Sigma(\eta_{0}%
,\mathbf{\gamma}_{0})^{-1}\mathbf{u})}{|\Sigma(\eta_{0},\mathbf{\gamma}%
_{0})|^{1/2}}, \label{denseli}%
\end{equation}
where $f_{0}^{\ast}$ is non increasing and is strictly decreasing in a
neighborhood of 0.

\item[A8] Let $H_{0}$ be the distribution of $\mathbf{x.}$ Then for any
$\mathbf{\delta}\in\mathbb{R}^{k}$, $\mathbf{\delta} \neq\mathbf{0}$ we have
$P_{H_{0}}(\mathbf{x}\mathbf{\delta} \neq\mathbf{0})>0$

\item[A9] (Identification condition). If $\mathbf{\gamma}\neq\mathbf{\gamma
}^{\ast}$ for all $\alpha$ we have $\Sigma(1,\mathbf{\gamma}) \neq
\Sigma(\alpha,\mathbf{\gamma}^{\ast})$.
\end{description}

An important family of distributions satisfying A7 is the multivariate normal,
in this case,
\begin{equation}
f_{0}^{\ast}(z)=(2\pi)^{-p/2}\exp(-z/2). \label{dennor}%
\end{equation}
Note that when the $(\mathbf{y}_{i},\mathbf{x}_{i})$s satisfy (\ref{mu}),
(\ref{cov}) and (\ref{cond}), A7 is satisfied. The following Theorem states
the consistency of composite $\tau$-estimators.

\begin{theorem}
\label{teo:consistency} Let $\mathbf{T}$ $=(\mathbf{t}_{1},\ldots
,\mathbf{t}_{n})$, $\mathbf{t}_{i}=(\mathbf{y}_{i},\mathbf{x}_{i})$, $1\leq
i\leq n$, be i.i.d random vectors with distribution $F_{0}$ and call $H_{0}$
the marginal distribution of the $\mathbf{x}_{i}s.$ Assume (i) $\rho_{1}$
satisfies (A1-A5), (ii) $\rho_{2}$ satisfies A1-A6, (iii) under $F_{0}$ A7 and
A8 holds and (iv) A9. Then, the composite $\tau$-estimators $\widehat
{\mathbf{\beta}}$, $\widehat{\mathbf{\gamma}}$ and $\widehat{\eta}$ satisfy
$\lim_{n\rightarrow\infty}\widehat{\mathbf{\beta}}=\mathbf{\beta}_{0}$ (a.s.),
$\lim_{n\rightarrow\infty}\widehat{\mathbf{\gamma}}=\mathbf{\gamma}_{0}$
(a.s.). Moreover, if $f_{0}^{\ast}$ is given by (\ref{dennor}) we also have
$\lim_{n\rightarrow\infty}\widehat{\eta}=\eta_{0}$ (a.s.).
\end{theorem}

Note that for the consistency of $\widehat{\mathbf{\beta}}$ and $\widehat
{\mathbf{\gamma}}$ is not necessary that $\mathbf{y}_{i}|\mathbf{x}_{i}$ be
multivariate normal. We do not give a formal proof of this Theorem. In Theorem
\ref{theorem:fisher} of the Appendix \ref{app:asymptotic} we give a rigorous
proof of the Fisher consistency of the estimating functional associated to the
compose $\tau$-estimator. From this result we derive an heuristic proof of
Theorem \ref{teo:consistency}.

The following Theorem states the asymptotic normality of composite $\tau
$-estimators. We need the following assumptions

\begin{description}
\item[A10] Let $H_{0}$ be the distribution of $\mathbf{x}$. Then $H_{0}$ has
finite second moments and $E_{H_{0}}(\mathbf{x} \mathbf{x}^{\top})$ is non--singular.

\item[A11] The functions $\rho_{i}$, $i=1,2$ are twice differentiable.
\end{description}

\begin{theorem}
\label{teo:normality} Let $\mathbf{\lambda}=(\mathbf{\beta}^{\top
},\mathbf{\gamma}^{\top})^{\top}$ and $\widehat{\mathbf{\lambda}}%
=(\widehat{\mathbf{\beta}}^{\top},\widehat{\mathbf{\gamma}}^{\top})^{\top}$ be
the composite $\tau$-estimator. Consider the same assumptions as in Theorem
\ref{teo:consistency}, A10 and A11. Then, we have
\begin{equation*}
\sqrt{n}(\widehat{\mathbf{\lambda}}-\mathbf{\lambda})\overset{D}{\rightarrow
}N\left(  \mathbf{0},\Sigma_{\mathbf{\lambda}}\right)  ,
\end{equation*}
where
\begin{equation*}
\Sigma_{\mathbf{\lambda}}=E\left[  \nabla_{\mathbf{\lambda}}^{2}%
T(\mathbf{\lambda})\right]  ^{-1}E\left[  \nabla_{\mathbf{\lambda}%
}T(\mathbf{\lambda})\text{ }\nabla_{\mathbf{\lambda}}T(\mathbf{\lambda
})^{\text{T}}\right]  \left(  E\left[  \nabla_{\mathbf{\lambda}}%
^{2}T(\mathbf{\lambda})\right]  ^{-1}\right)  ^{\text{T}},
\end{equation*}
and $\nabla_{\mathbf{\lambda}}T(\mathbf{\lambda})$ and $\nabla
_{\mathbf{\lambda}}^{2}T(\mathbf{\lambda})$ are the gradient and Hessian
matrix of $T(\mathbf{\lambda})$ respectively.
\end{theorem}

We do not give the proof of Theorem \ref{teo:normality}. However, it can be
obtained using standard delta method arguments, see for example Theorem 10.9
in \citet{maronna_martin_yohai_2006}. This Theorem allows to define Wald tests
for null hypothesis and confidence intervals for $\mathbf{\beta}$ and
$\mathbf{\gamma}$, but not for $\eta$. However in most practical applications
the interest is centered in $\mathbf{\beta}$ and $\mathbf{\gamma}$.

\section{Computational aspects \label{sec:computation}}

The composite $\tau$-estimators are obtained by an iterative algorithm.
Hereafter we provide some details. Given starting values $\widetilde
{\mathbf{\beta}}^{(0)}$ and $\widetilde{\mathbf{\gamma}}^{(0)}$, we perform
iterations on steps \textbf{(A)-(C)} until convergence. Suppose that we have
already computed $\widetilde{\mathbf{\beta}}^{(h)}$ and $\widetilde
{\mathbf{\gamma}}^{(h)},$ then we performed the following steps \textbf{A},
\textbf{B} and \textbf{C} to obtain $\widetilde{\mathbf{\beta}}^{(h+1)}$ and
$\widetilde{\mathbf{\gamma}}^{(h+1)}$:

\begin{description}
\item[(A)] Find scales $s_{jl}(\widetilde{\mathbf{\beta}}^{(h)},\widetilde
{\mathbf{\gamma}}^{(h)})$ ($1\leq j<l\leq p$) by solving equations
\begin{equation*}
\frac{1}{n}\sum_{i=1}^{n}\rho_{1}\left(  \frac{m_{i}^{jl}(\widetilde
{\mathbf{\beta}}^{(h)},\widetilde{\mathbf{\gamma}}^{(h)})}{s_{jl}%
(\widetilde{\mathbf{\beta}}^{(h)},\widetilde{\mathbf{\gamma}}^{(h)})}\right)
=b
\end{equation*}

\item[(B)] Update $\mathbf{\beta}$ by the fixed point equation using equation
(\ref{equ:Tbeta}) derived in Appendix \ref{sec:equationsBeta}. That is,
\begin{equation*}
\widetilde{\mathbf{\beta}}^{(h+1)}=\left[  \sum_{i=1}^{n}\sum_{j=1}^{p-1}%
\sum_{l=j+1}^{p}\tilde{W}_{i}^{jl}(\widetilde{\mathbf{\beta}}^{(h)}%
,\widetilde{\mathbf{\gamma}}^{(h)})\left(  \dot{\mathbf{x}}_{i}^{jl\top}%
\dot{\mathbf{x}}_{i}^{jl}\right)  \right]  ^{-1}\sum_{i=1}^{n}\sum_{j=1}%
^{p-1}\sum_{l=j+1}^{p}\tilde{W}_{i}^{jl}(\widetilde{\mathbf{\beta}}%
^{(h)},\widetilde{\mathbf{\gamma}}^{(h)})\left(  \dot{\mathbf{x}}_{i}^{jl\top
}\dot{\mathbf{y}}_{i}^{jl}\right)  ,
\end{equation*}
where $\dot{\mathbf{x}}_{i}^{jl}=\Sigma_{jl}^{\ast-1/2}\mathbf{x}_{i}^{jl}.$

\item[(C)] A fixed point equation for $\mathbf{\gamma}$ can be derived from
the estimating equation (\ref{equ:Tgamma}). However we found that to use this
equation to update $\mathbf{\gamma}$ was numerically unstable. We preferred to
make this updating by means of a direct minimization of the goal function $T$
defined in (\ref{goal-tau}), that is, we define $\widetilde{\mathbf{\gamma}%
}^{(h+1)}$ by
\begin{equation*}
\widetilde{\mathbf{\gamma}}^{(h+1)}=\arg\min_{\mathbf{\gamma}} T(\widetilde
{\mathbf{\beta}}^{(h+1)},\mathbf{\gamma}).
\end{equation*}
For this purpose, in the example of Section \ref{sec:examples} and in the
Monte Carlo study of Section \ref{sec:simulations} we used the function
\texttt{optim} of the \texttt{R} program.

\item[(D)] Once the convergence criterion for $(\widetilde{\mathbf{\beta}%
}^{(h)},\mathbf{\gamma}^{(h)})$ is reached, the estimator of $\eta$ is
obtained by solving the equation (\ref{equ:eta0}).
\end{description}

To start the iterative algorithm the initial estimators $\widetilde
{\mathbf{\beta}}^{(0)}$ and $\widetilde{\mathbf{\gamma}}^{(0)}$ are necessary.
Let $\mathbf{Y}=(\mathbf{y}_{1}^{\top},\ldots,\mathbf{y}_{n}^{\top})^{\top}$,
$\mathbf{x}_{i}^{(j)}$ the $j$-th column of $\mathbf{x}_{i},$ $\mathbf{X}%
_{j}=(\mathbf{x}_{1}^{(j)\top},\ldots,\mathbf{x}_{n}^{(j)\top})^{\top}$,
$1\leq j\leq k$. Then, $\widetilde{\mathbf{\beta}}^{(0)}$can be obtained by
means of robust regression estimator using $\mathbf{Y}$ as response and
$\mathbf{X}_{j}$, $1\leq j\leq k$ as covariables. In our algorithm we use an
MM-estimator as implemented in function \texttt{lmRob} of the \texttt{R} package
\texttt{robust} using an efficiency of $0.85$. Once this initial estimator
$\widetilde{\mathbf{\beta}}^{(0)}$ is computed, the residuals $\mathbf{r}%
_{i}=\mathbf{y}_{i}-\mathbf{x}_{i}\widetilde{\mathbf{\beta}}^{(0)}$
$(i=1,\ldots,n)$ can be evaluated. Then, a robust covariance matrix of
$\mathbf{u}$ robust under the ICM model is obtained applying to these
residuals  the estimator  presented in
\citet{agostinelli_leung_yohai_zamar_2014} based on filtering and
S-estimators with missing observations. Call $\widetilde{\mathbf{\Sigma}}$ to
this matrix and let $\mathbf{t}$ be the vector of the $p(p+1)/2$ values of the
lower triangular side of this matrix including the diagonal elements. In a
similar way, let $\mathbf{v}_{j}$ be the column vector of the $p(p+1)/2\times
1$ values of the lower triangular side of the matrix $V_{j}$. An initial
estimator $\widetilde{\mathbf{\gamma}}^{(0)}$ of $\mathbf{\gamma}$ could be
obtained by means of a regression estimator using $\mathbf{t}$ as response and
$\mathbf{v}_{1},\ldots,\mathbf{v}_{J}$ as covariables. Since neither
$\mathbf{t}$ nor $\mathbf{v}_{1},\ldots,\mathbf{v}_{J}$ need to have outliers,
it is not necessary to use a robust estimator, in fact we use function
\texttt{lm} of \texttt{R} to perform this step.

\section{Example \label{sec:examples}}

Hereafter we present one application of the introduced method on a real
data set. The example is a prospective longitudinal study of
children with disorder of neural development. In this data set, outliers are
present in the couples rather than in the units and the composite $\tau$-estimator provides a different analysis with respect to maximum likelihood
and classical robust procedures.

\subsection{Autism \label{sec:autism}}

The data used in this example were collected by researchers at the University
of Michigan \citep{anderson_oti_lord_welch_2009} as part of a prospective
longitudinal study of $214$ children and they are analyzed, among others, also
in \citet{west_welch_galechi_2007}. The children were divided into three
diagnostic groups when they were $2$ years old: autism, pervasive
developmental disorder (PDD), and nonspectrum children. The study was designed
to collect information on each child at ages $2$, $3$, $5$, $9$, and $13$
years, although not all children were measured at each age. One of the study
objectives was to assess the relative influence of the initial diagnostic
category (autism or PDD), language proficiency at age $2$, and other
covariates on the developmental trajectories of the socialization of these
children. Study participants were children who had consecutive referrals to
one of two autism clinics before the age of $3$ years. Social development was
assessed at each age using the Vineland Adaptive Behavior Interview survey
form, a parent-reported measure of socialization. The dependent variable, vsae
(Vineland Socialization Age Equivalent), was a combined score that included
assessments of interpersonal relationships, play/leisure time activities, and
coping skills. Initial language development was assessed using the Sequenced
Inventory of Communication Development (SICD) scale; children were placed into
one of three groups (sicdegp, $s_{(1)},s_{(2)},s_{(3)}$, where $s_{(k)}$ is
the indicator function of the $k$ group) based on their initial SICD scores on
the expressive language subscale at age $2$. We consider the subset of $n=41$
children for which all measurements are available. We analyze this data using
a regression model with random coefficients where vsae is explained by
intercept, age, age$^{2}$ and sicdegp as a factor variable plus interaction
among the age related variables and sicdegp. Hereafter, the variable age is
shifted by $2$. Let $y_{ij}$ be the value of the $i$-th vsae for the $j$-th
ages value $a_{j}$, then it is assumed that for $1\leq i\leq41,$ $1\leq
j\leq5$ we have%
\begin{align*}
y_{ij}  &  =b_{i1}+b_{i2}a_{j}+b_{i3}a_{j}^{2}\\
&  +\beta_{4}s_{(1)i}+\beta_{5}s_{(2)i}\\
&  +\beta_{6}a_{j}\times s_{(1)i}+\beta_{7}a_{j}\times s_{(2)i}+\beta_{8}%
a_{j}^{2}\times s_{(1)i}+\beta_{9}a_{j}^{2}\times s_{(2)i}+\varepsilon_{ij},
\end{align*}
where $(b_{i1},b_{i2},b_{i3})$ are i.i.d. random coefficients with mean
$(\beta_{1},\beta_{2},\beta_{3})$ and covariance matrix
\begin{equation*}
\Sigma_{b}=\left(
\begin{array}
[c]{ccc}%
\sigma_{11} & \sigma_{1a} & \sigma_{1a^{2}}\\
\sigma_{1a} & \sigma_{aa} & \sigma_{aa^{2}}\\
\sigma_{1a^{2}} & \sigma_{aa^{2}} & \sigma_{a^{2}a^{2}}\\
&  &
\end{array}
\right)  .
\end{equation*}
$\beta_{4},\ldots,\beta_{9}$ are fixed coefficients and the $\varepsilon_{ij}$
are i.i.d. random errors independent of the random coefficients with zero mean
and variance $\sigma_{\varepsilon\varepsilon}$. Then, the model could be
rewritten in term of (\ref{mu}) and (\ref{cov}) with $p=5$, $n=41$, $J=6$ and
$k=9$, $\mathbf{y}_{i}=(y_{i1}, \ldots,y_{i5})^{\top}$,
\begin{equation*}
\mathbf{x}_{i}=\left(
\begin{array}
[c]{ccccccccc}%
1 & a_{1} & a_{1}^{2} & s_{(1)i} & s_{(2)i} & a_{1}\ s_{(1)i} &
a_{1}\ s_{(2)i} & a_{1}^{2}\ s_{(1)i} & a_{1}^{2}\ s_{(2)i}\\
 \cdots &  \cdots &  \cdots &  \cdots &  \cdots &  \cdots &  \cdots &  \cdots &  \cdots\\
 \cdots &  \cdots &  \cdots &  \cdots &  \cdots &  \cdots &  \cdots &  \cdots &  \cdots\\
1 & a_{5} & a_{5}^{2} & s_{(1)i} & s_{(2)i} & a_{5}\ s_{(1)i} &
a_{5}\ s_{(2)i} & a_{5}^{2}\ s_{(1)i} & a_{5}^{2}\ s_{(2)i}
\end{array}
\right)  ,
\end{equation*}
while the variance and covariance structure $\Sigma(\eta,\mathbf{\gamma}%
)=\eta(I+\sum_{j=1}^{J}\gamma_{j}V_{j})$ is as follows. Let, $\mathbf{j}$ a
$5$-vector of ones, $\mathbf{a}=(a_{1},a_{2},a_{3},a_{4},a_{5})^{\top}$, which
corresponds to $\text{age}$ and $\mathbf{b}=\mathbf{a}^{2}$ which corresponds
to $\text{age}^{2}$. Then, we have $V_{1}=\mathbf{j}\mathbf{j}^{\top}$,
$V_{2}=\mathbf{a}\mathbf{a}^{\top}$, $V_{3}=\mathbf{b}\mathbf{b}^{\top}$,
$V_{4}=\mathbf{j}\mathbf{a}^{\top}+\mathbf{a}\mathbf{j}^{\top}$,
$V_{5}=\mathbf{j}\mathbf{b}^{\top}+\mathbf{b}\mathbf{j}^{\top}$ and
$V_{6}=\mathbf{a}\mathbf{b}^{\top}+\mathbf{b}\mathbf{a}^{\top}$. $\eta
=\sigma_{\varepsilon\varepsilon}$ is the scale of the error term, $\gamma
_{1}=\sigma_{11}/\sigma_{\varepsilon\varepsilon}$, $\gamma_{2}=\sigma
_{aa}/\sigma_{\varepsilon\varepsilon}$, $\gamma_{3}=\sigma_{a^{2}a^{2}}%
/\sigma_{\varepsilon\varepsilon}$, $\gamma_{4}=\sigma_{1a}/\sigma
_{\varepsilon\varepsilon}$, $\gamma_{5}=\sigma_{1a^{2}}/\sigma_{\varepsilon
\varepsilon}$ and $\gamma_{6}=\sigma_{aa^{2}}/\sigma_{\varepsilon\varepsilon}$.

\begin{table}[ptb]
\begin{center}
\resizebox{\textwidth}{!}{
\begin{tabular}{lrrrrrrrrr}
\hline
Method & Int. & $a$ & $a^{2}$ & $s_{(1)}$ & $s_{(2)}$ & $a \times s_{(1)}$ & $a \times s_{(2)}$ &  $a^2 \times s_{(1)}$ & $a^2 \times s_{(2)}$ \\ \hline
Max. Lik. &  12.847  &  6.851  & $-$0.062  & $-$5.245  & $-$2.154  &  $-$6.345  &  $-$4.512  &  0.133  &  0.236 \\
&  [0.000] & [0.000] &   [0.579] &   [0.041] &   [0.325] &    [0.000] &    [0.000] & [0.446] & [0.121] \\
Composite $\tau$ & 12.143 & 6.308 & $-$0.089 & $-$5.214 & $-$4.209 & $-$5.361 & $-$3.852 & 0.082 & 0.061 \\
& [0.000] & [0.000] & [0.329] & [0.000] & [0.012] & [0.000] & [0.001] & [0.578]& [0.677] \\
S Rocke & 10.934 & 7.162 & $-$0.107& $-$4.457& $-$0.108& $-$5.769& $-$4.995& 0.094 & 0.419 \\
& [0.000]  & [0.001]  & [0.666] & [0.049]  & [0.957] & [0.002] & [0.000] & [0.688]  & [0.011] \\
SMDM &   12.346&    6.020&   0.001 & $-$5.192 & $-$2.173 & $-$5.190 & $-$3.870 &   0.046 &   0.151 \\
& [0.000] & [0.000] & [0.992] & [0.010] & [0.213] & [0.000] & [0.000] & [0.781] & [0.300] \\
\hline
\end{tabular}}
\end{center}
\caption{Autism data set. Estimated fixed term parameters by different
methods. P-values are reported under squared parenthesis.}%
\label{tab:autism:fixed}%
\end{table}

Table \ref{tab:autism:fixed} report the estimators and the inference for the
fixed term parameters using different methods, while Table
\ref{tab:autism:random} reports the estimators of the random effect terms. ML,
S and SMDM provide similar results, while differences are present with
the composite $\tau$ method. Main differences are on the estimation of the
random effects terms, both in size (error variance component) and shape
(correlation components). Composite $\tau$ assign part of the total variance
to the random components while the other methods assign it to the error term. 
In fact, variances estimated by composite $\tau$ are
in general larger than that estimated by the other methods; composite $\tau$
suggests negative correlation between intercept and age, while ML, S and SMDM
suggest positive correlation. Composite $\tau$ provides small
estimates compared to the other methods for the error variance. 
These discrepancies reflects mainly
on the inference for the fixed term coefficients where the variable sicdegp is
significant using composite $\tau$ but is not using ML, S and SMDM procedures.
Interactions between age$^{2}$ and sicdegp is highly non significant using
composite $\tau$ and SMDM while it is somewhat significant using S.

To investigate more the reasons of differences between composite robust procedure and classic robust procedure results, we investigate cell, couple and row outliers. For a given dimension $1 \le q \le p$ we define as $q$-dimension outliers those $q$-dimension observations such that the corresponding squared Mahalanobis distance is greater than a quantile order $\alpha$ of a chi-square distribution with $q$ degree of freedom. In particular we call cell, couple and row outliers the $1$-dimension, $2$-dimension and $p$-dimension outliers respectively.
Composite $\tau$ procedure identifies $33$ couple outliers out of $410$ 
couples ($8\%$) at $\alpha=0.999$. The affected rows, with at least 
one couple outliers, are $12$ out of $41$. This means that the classic S 
and SMDM procedures have to deal with a data set with a level of 
contamination about $29\%$. We also run the analysis using the composite S 
estimator, not reported here, the results are similar to
those obtained by the composite $\tau$ estimator.

\begin{table}[ptb]
\begin{center}
\resizebox{\textwidth}{!}{
\begin{tabular}{lrrrrrrr}
\hline
Method & $\sigma_{11}$ & $\sigma_{aa}$ & $\sigma_{a^{2}a^{2}}$ & $\sigma_{1a} $ & $\sigma_{1a^{2}}$ & $\sigma_{aa^{2}}$ & $\sigma_{\varepsilon\varepsilon}$ \\ \hline
Maximum Likelihood & 2.643 & 2.328 & 0.102 & 0.775 & 0.429 & $-$0.038 & 51.360 \\
Composite $\tau$   & 9.362 & 9.670 & 0.052 & $-$4.019 & $-$0.002 & $-$0.327 & 5.164 \\
S Rocke            & 9.467 & 3.373 & 0.222 & 2.170 & 1.062 & $-$0.349 & 22.209 \\
SMDM    & 5.745 & 0.092 & 0.115 & 0.727 & 0.813 & 0.103 & 25.385 \\
\hline
\end{tabular}}
\end{center}
\caption{Autism data set. Estimated random term parameters by different
methods.}%
\label{tab:autism:random}%
\end{table}

\section{Monte Carlo simulations \label{sec:simulations}}

In this section we describe the results of a Monte Carlo study with the aim of
illustrating the performance of the new procedure in the classical
contamination Model (CCM) and in the independent contamination model (ICM). We
consider a 2-way crossed classification with interaction linear mixed model
\begin{equation}
y_{fgh}=\mathbf{x}_{fgh}^{\top}\mathbf{\beta}_{0}+a_{f}+b_{g}+c_{fg}+e_{fgh},
\end{equation}
where $f=1,\ldots,F$, $g=1,\ldots,G$, and $h=1,\ldots,H$. Here, we set $F=2$,
$G=2$ and $H=3$ which leads to $p=F\times G\times H=12$. $\mathbf{x}_{fgh}$ is
a $(k+1)\times1$ vector where the last $k$ components are from a standard
multivariate normal and the first component is identically equal to $1$,
$\mathbf{\beta}_{0}=(0,2,2,2,2,2)^{\top}$ is $(k+1)\times1$ vector of the
fixed parameters with $k=5$. $a_{f}$, $b_{g}$ and $c_{fg}$ are the random
effect parameters which are normally distributed with variances $\sigma
_{a}^{2}$, $\sigma_{b}^{2}$, and $\sigma_{c}^{2}$. Arranging the $y_{fgh}$ in
lexicon order (ordered by $h$ within $g$ within $f$) we obtain the vector
$\mathbf{y}$ of dimension $p$ and in the similar way the $p\times k$ matrix
$\mathbf{x}$ obtained arranging $\mathbf{x}_{fgh}$. Similarly, we let
$\mathbf{a}=(a_{1},\ldots,a_{F})^{\top}$, $\mathbf{b}=(b_{1},\ldots
,b_{F})^{\top}$ and $\mathbf{c}=(c_{11},\ldots,c_{FG})^{\top}$, that is,
$\mathbf{a}\sim N_{F}(\mathbf{0},\sigma_{a}^{2}I_{F})$ and similar for
$\mathbf{b}$ and $\mathbf{c}$, while $\mathbf{e}=(e_{111},\ldots
,e_{FGH})^{\top}\sim N_{p}(\mathbf{0},\sigma_{e}^{2}I_{p})$. Hence
$\mathbf{y}$ is a $p$ multivariate normal with mean $\mathbf{\mu
}=\mathbf{x\beta}$ and variance matrix $\Sigma_{0}=\eta_{0}(V_{0}+\sum
_{j=0}^{J}\gamma_{j}V_{j}),$ where $V_{0}=I_{p}$, $V_{1}=I_{F}\otimes
J_{G}\otimes J_{H}$, $V_{2}=J_{F}\otimes I_{G}\otimes J_{H}$, and $V_{3}%
=J_{F}\otimes J_{G}\otimes I_{H}$; $\otimes$ is the Kronecker product and $J$
is a matrix of ones with appropriate dimension. We took $\sigma_{e}^{2}%
=\sigma_{a}^{2}=\sigma_{b}^{2}=1$ and $\sigma_{c}^{2}=2.$ Then $\mathbf{\gamma
}_{0}=(\gamma_{10},\gamma_{20},\gamma_{30})^{\top}=(\sigma_{a}^{2}/\sigma
_{e}^{2},\sigma_{b}^{2}/\sigma_{e}^{2},\sigma_{c}^{2}/\sigma_{e}^{2})^{\top
}=(1/4,1/4,1/2)^{\top}$ and $\eta_{0}=\sigma_{e}^{2}=1/4$. We consider a
sample of size $n=100$ and four levels of contamination $\varepsilon=0,5,10$
and $15\%$. In the CCM $n\times\varepsilon$ observations are contaminated by
replacing all the elements of the vector $\mathbf{y}$ by observations from
$\mathbf{y}_{0}\sim N_{p}(\mathbf{x}_{0}\mathbf{\beta}_{0}+\mathbf{\omega}%
_{0},\Sigma)$ where the corresponding components of $\mathbf{x}$ are sampled
from $\mathbf{x}_{0}\sim N_{p\times k}(\mathbf{\lambda}_{0},0.005^{2}%
\mathbf{I}_{p\times k})$ with all the components of $\mathbf{\lambda}_{0}$
equal to $1$ in the case of low leverage outliers (lev1) or to $20$ for large
leverage outliers (lev20) and $\mathbf{\omega}_{0}$ is a $p$-vector of
constants all equal to $\omega_{0}$. In the ICM we replace $n\times
p\times\varepsilon$ cells, randomly chosen in the $n\times p=1200$ values of
the dependent variable by $\mathbf{y}_{0}$ and the corresponding $k$ vector of
the explanatory variables by $\mathbf{x}_{0}$ with values as in the previous
case. In each cases we move $\omega_{0}$ trying at attain the maximum MSE. For
each combination of these factors we run the S-estimator as described in
\citet{victoria-feser_high_2006} with $\rho$ function with asymptotic
rejection probability set to $0.01$. We compute the composite $\tau$-estimator
with $\rho_{1}$ and $\rho_{2}$ in the family given by (\ref{optimalsq}) with
constant $c$ equals to $1.64$. For each case we run $500$ Monte Carlo
replications. 

Let $(\mathbf{y},\mathbf{x})$ be an observation independent of the sample
$(\mathbf{y}_{1}, \mathbf{x}_{1}), \ldots,(\mathbf{y}_{n}, \mathbf{x}_{n})$ 
used to compute $\widehat{\mathbf{\beta}}$ and let
$\widehat{\mathbf{y}}=\mathbf{x} \widehat{\mathbf{\beta}}$ be the predicted
value of $\mathbf{y}$ using $\mathbf{x}.$ Then the square Mahalanobis
distance between $\widehat{\mathbf{y}}$ and $\mathbf{y}$ using the matrix
$\Sigma_{0}$ is
\begin{align*}
m(\widehat{\mathbf{y}}, \mathbf{y}, \Sigma_{0}) &  =(\widehat{\mathbf{y}%
}-\mathbf{y})^{\top} \Sigma_{0}^{-1} (\widehat{\mathbf{y}}-\mathbf{y}) \\
& =(\widehat{\mathbf{\beta}}-\mathbf{\beta}_{0})^{\top} \mathbf{x}^{\top} 
\Sigma_{0}^{-1} \mathbf{x} (\widehat{\mathbf{\beta}}-\mathbf{\beta}_{0}) \\
& + (\mathbf{y}-\mathbf{x} \mathbf{\beta}_{0})^{\top} \Sigma_{0}^{-1} 
(\mathbf{y}-\mathbf{x} \mathbf{\beta}_{0}).
\end{align*}
Since $\mathbf{y}-\mathbf{x}\mathbf{\beta}_{0}$ is independent of $\mathbf{x}$ and has
covariance matrix $\mathbf{\Sigma}_{0}$, putting $A = E(\mathbf{x}^{\top} \Sigma_{0}^{-1} \mathbf{x})$ 
we have
\begin{align*}
E \left[ m(\widehat{\mathbf{y}}, \mathbf{y}, \Sigma_{0}) \right] 
& = E\left[ (\widehat{\mathbf{\beta}} - \mathbf{\beta}_{0})^{\top} A (\widehat{\mathbf{\beta}} - \mathbf{\beta}_{0}) \right] \\ 
& + \operatorname{trace} (\Sigma_{0}^{-1} (\mathbf{y}- \mathbf{x} \mathbf{\beta}_{0})(\mathbf{y}- \mathbf{x} \mathbf{\beta}_{0})^{\top}) \\
& = E\left[ (\widehat{\mathbf{\beta}}-\mathbf{\beta}_{0})^{\top} A (\widehat{\mathbf{\beta}}-\mathbf{\beta}_{0})\right] + p.
\end{align*}
Then, to evaluate an estimator $\widehat{\mathbf{\beta}}$ of $\mathbf{\beta}$ 
by its prediction performance we can use
\begin{equation}
E\left[ m(\widehat{\mathbf{\beta}},\mathbf{\mathbf{\beta}}_{0},A) \right]
=E \left[ (\widehat{\mathbf{\beta}} - \mathbf{\beta}_{0})^{\top} A (\widehat{\mathbf{\beta}}-\mathbf{\beta}_{0})\right] .
\end{equation}
Let $N$ be the number of replications in the simulation study, and let
$\widehat{\mathbf{\beta}}_{j}$, $1\leq j\leq N$ be the value of
$\widehat{\mathbf{\beta}}$ at the $j$-th replication, then we can estimate
$E\left[  m(\widehat{\mathbf{\beta}},\mathbf{\mathbf{\beta}}_{0},A) \right]$
by the Mean Square Mahalanobis distance
\begin{equation*}
\text{MSMD} = \frac{1}{N} \sum_{j=1}^{N} m(\widehat{\mathbf{\beta}}_{j}, \mathbf{\beta}_{0},A).
\end{equation*}
It is easy to prove that as in this case $\mathbf{x}$ is a $p\times k$ matrix
where the cells are independent $N(0,1)$ random variables, then $A= \operatorname{trace}(\Sigma_{0}^{-1}) I_{k}$.

Given two covariance matrices $\Sigma_{1}$ and $\Sigma_{0},$ one way to
measure how close are $\Sigma_{1}$ and $\Sigma_{0}$ is by the
Kullback-Leibler divergence between two normal distributions with the same
mean and covariance matrices equal to $\Sigma_{1}$ and $\Sigma_{0}$ given by%
\begin{equation*}
\text{KLD}(\Sigma_{1},\Sigma_{0})=\text{trace}\left(  \Sigma_{1}\Sigma
^{-1}\right)  -\log\left(  \Sigma_{1}\Sigma_{0}^{-1}\right)  -p.
\end{equation*}
Since $(\eta_{0},\mathbf{\gamma}_{0})$ determines $\Sigma_{0}=\Sigma
(\eta_{0},\mathbf{\gamma}_{0}),$ that is, the covariance matrix of
$\mathbf{y}$ given $\mathbf{x}$, one way to measure the performance of an
estimator $(\widehat{\eta},\widehat{\mathbf{\gamma}})$ of $(\eta
_{0},\mathbf{\gamma}_{0})$ is by
\begin{equation*}
E\left[ \text{KLD}(\Sigma(\widehat{\eta},\widehat{\mathbf{\gamma}}),\mathbf{\Sigma}_{0})\right] .
\end{equation*}
Let $(\widehat{\eta}_{j},\widehat{\mathbf{\gamma}}_{j}),1\leq j\leq N$, be
the value of $(\widehat{\eta},\widehat{\mathbf{\gamma}})$ at the $j$-th
replication, then we can estimate $E\left[ \text{KLD}(\Sigma(\widehat{\eta}%
,\widehat{\mathbf{\gamma}}),\mathbf{\Sigma}_{0})\right]$ by the Mean
Kullback-Leibler Divergence
\begin{equation*}
\text{MKLD}=\frac{1}{N}\sum_{j=1}^{N} \text{KLD}(\Sigma(\widehat{\eta}_{j}, \widehat{\mathbf{\gamma}}_{j}),\mathbf{\Sigma}_{0}).
\end{equation*}

Table \ref{tab:efficiency} reports the relative efficiency of the classic S-
and composite $\tau$-estimators with respect to the maximum likelihood in
absence of contamination. The efficiency of estimators of $\mathbf{\beta
}_{0}$ will be measured for the MSMD ratio while the efficiency of an
estimator of $(\eta_{0},\mathbf{\gamma}_{0})$ by the MKLD ratio.
\begin{table}[ptb]
\begin{center}%
\begin{tabular}
[c]{lcc}\hline
Method & MSMD EFF. & MKLD EFF.\\\hline
S & 0.712 & 0.571\\
composite $\tau$ & 0.806 & 0.739\\\hline
\end{tabular}
\end{center}
\caption{Relative efficiency of S- and composite $\tau$- estimators}%
\label{tab:efficiency}%
\end{table}

We report the results under $10\%$ of both types outlier contamination:
classical and independent. Figure \ref{fig:mse-beta-paper10} reports the
behavior of the MSMD as a function of $\omega_{0}$ while Figure
\ref{fig:mse-sigma-paper10} reports the behavior of MKLD. For easy of
comparison, Table \ref{tab:max} reports the maximum values of MSMD and MKLD in
the range of the Monte Carlo setting. Since similar behavior is observed for
negative values of $\omega_{0}$, these results are not reported.

Similar behavior was observed for the case $5\%$ and $15\%$ which are not
reported here. The composite $\tau$-estimator is very competitive with the
classical S-estimator under the classical contamination model, in fact, in the
low leverage case (lev1) the maximum values of MSMD and MKLD of the
composite $\tau$-estimator are only slightly larger than those of the
S-estimator. Instead for the high leverage case (lev20) the values MSDM are of
essentially the same for both estimators, while the maximum value of MKLD
is smaller for the composite $\tau$-estimator. In the independent
contamination model the composite $\tau$-estimator clearly outperforms the
classical S-estimator. In fact, while the MSMD and MKLD of the composite
$\tau$-estimator are always bounded by a small value, the MSMD and MKLD of the
classical S always show an unbounded behavior.

\begin{figure}[ptb]
\begin{center}
\includegraphics[width=0.8\textwidth]{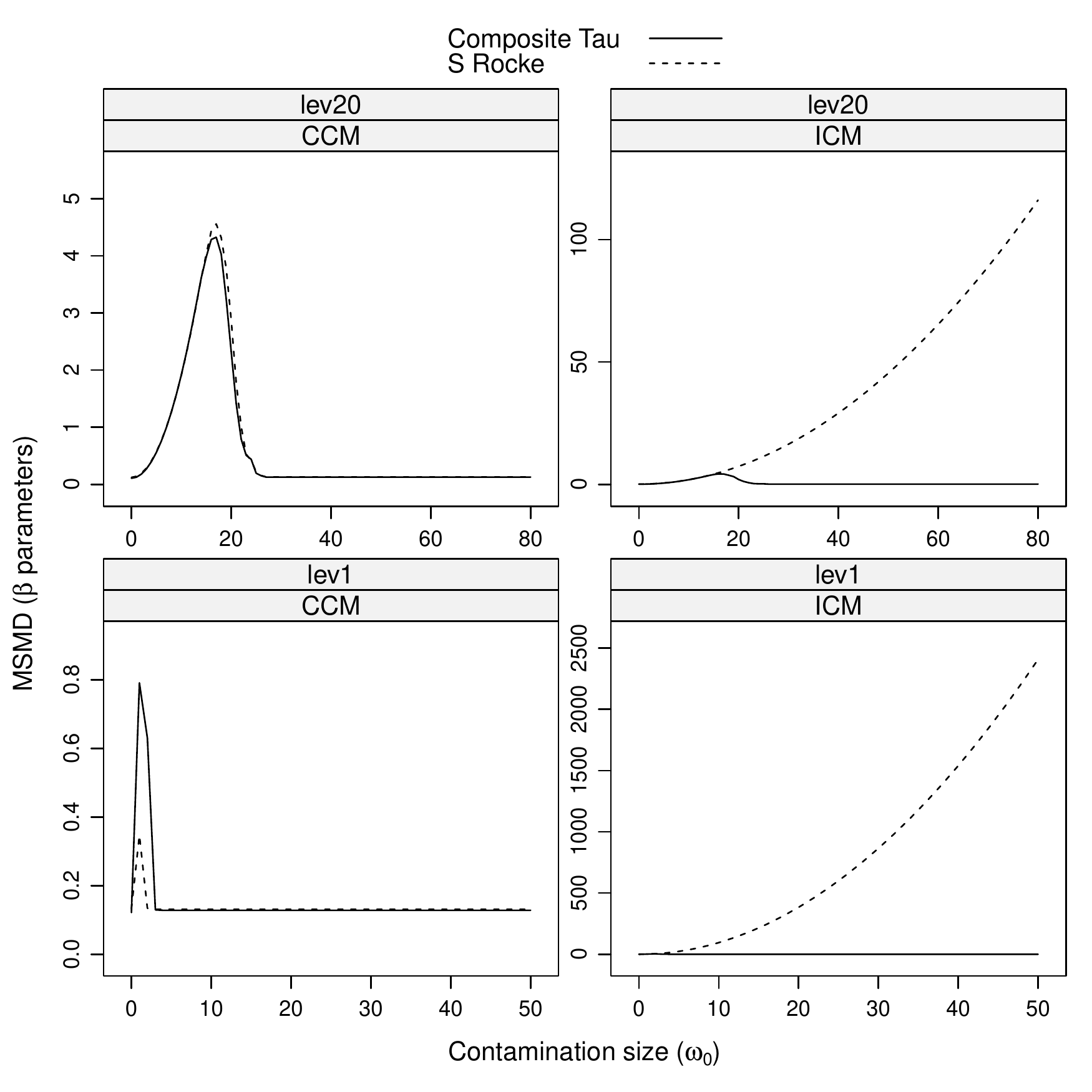}
\end{center}
\caption{MSMD performance of the S- and composite $\tau$-estimators of
$\mathbf{\beta_{0}}$ under 10\% of outlier contamination}%
\label{fig:mse-beta-paper10}%
\end{figure}

\begin{figure}[ptb]
\begin{center}
\includegraphics[width=0.8\textwidth]{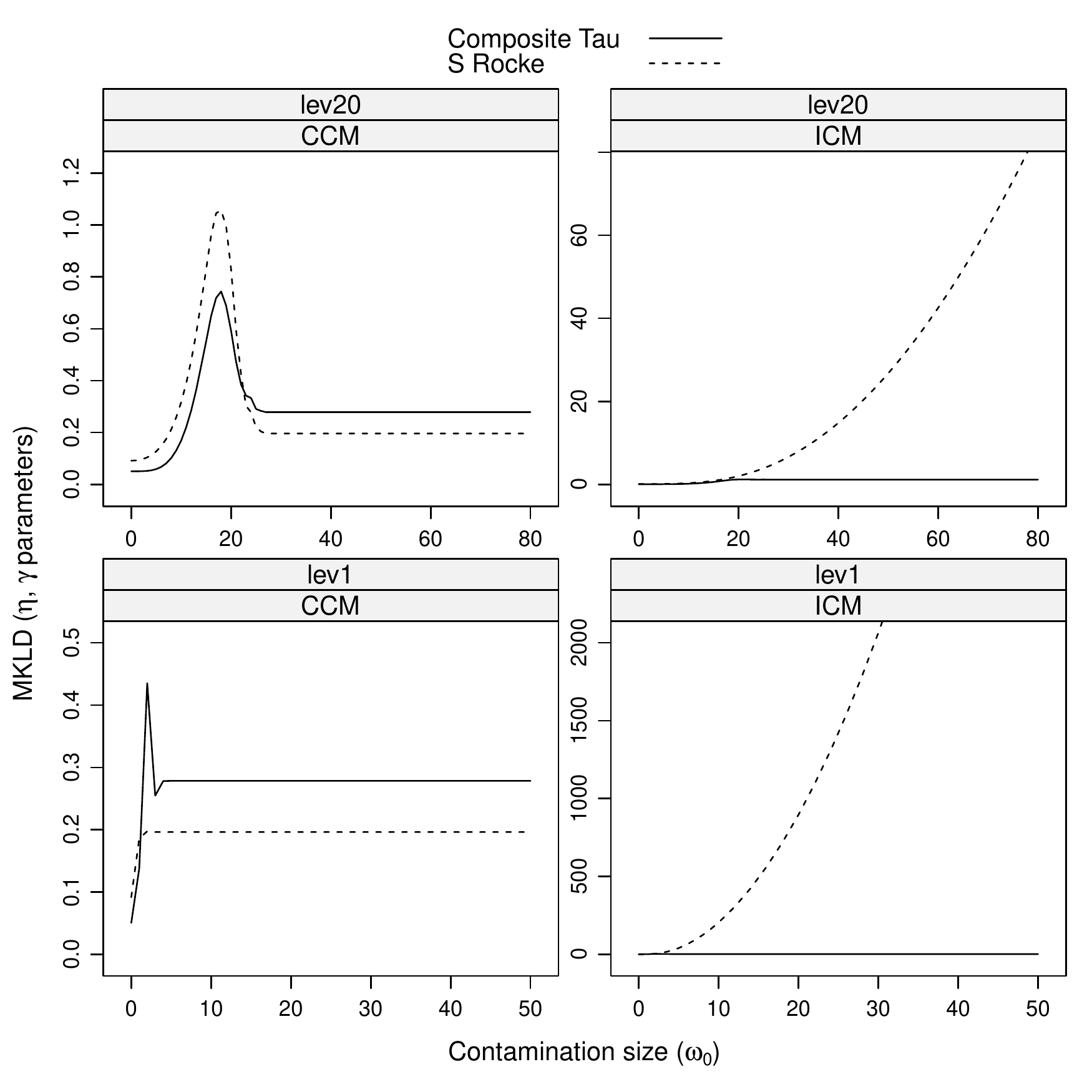}
\end{center}
\caption{MKLD performance of the S- and composite $\tau$-estimators of
$(\eta_{0},\mathbf{\gamma_{0}})$ under 10\% of outlier contamination }%
\label{fig:mse-sigma-paper10}%
\end{figure}

\begin{table}[ptb]
\begin{center}%
\begin{tabular}[c]{lcrr@{\extracolsep{15pt}}rr}\hline
&  & \multicolumn{2}{c}{CCM} & \multicolumn{2}{c}{ICM}\\
\cline{3-4} \cline{5-6} 
& Method & lev1 & lev20 & lev1 & lev20\\
\hline
MSDM & S & 0.347 & 4.558 & 2406.850 & 116.080 \\
& composite $\tau$ & 0.791 & 4.325 & 3.170 & 4.292 \\
\hline
MKLD & S & 0.197 & 1.057 & 5819.794 & 85.281 \\
& composite $\tau$ & 0.435 & 0.744 & 2.086 & 1.204 \\
\hline
\end{tabular}
\end{center}
\caption{Maximum values of MSDM and MKLD in Figures 1 and 2 respectively for
S- and composite $\tau$-estimators}%
\label{tab:max}%
\end{table}

\clearpage

\section{Conclusions \label{sec:conclusions}}

The independent contamination model presents new challenge problems for robust
statistics. Robust estimators developed for the classical Tukey-Huber
contamination model show non robust behavior under the ICM, in particular
their breakdown point converges to zero as the dimension $p$ increases.
Furthermore, affine equivariance, a proven asset for achieving CCM robustness,
becomes a hindrance under ICM because of outliers propagation. We introduce a
new class of robust estimators namely composite S-estimators and composite
$\tau$-estimators which are based on M and $\tau$-scales of the squared
Mahalanobis distances of two dimensional subvectors of $\mathbf{y}$ using the
same idea from the composite likelihood. We apply them in linear mixed models
estimation. Our methods are equivariant for some selected transformations
nevertheless provide fairly high resistance against both CCM and ICM outliers
with breakdown point $0.5$ and $0.25$ respectively.

An \texttt{R} \citep{Rsoftware} package \texttt{robustvarComp} is available in the Comprehensive R Archive Network at \\ \href{http://cran.r-project.org/web/packages/robustvarComp/index.html}{cran.r-project.org/web/packages/robustvarComp/index.html}. The package implements composite S and $\tau$-estimators and the classic S estimator for linear mixed models.

\section*{Acknowledgements} 

Victor Yohai research was partially supported by Grants PIP 112-200801-00216 
and 112-201101-00339 from CONICET, PICT 2011-397 from ANPCYT and W276 from 
the University of Buenos Aires, Argentina.

\appendix

\section{Estimating Equations and Algorithms}

In this Appendix we provide the derivative of the loss function
$T(\mathbf{\beta},\mathbf{\gamma})$ with respect to $\mathbf{\beta}$ and
$\mathbf{\gamma}$. For $\mathbf{\beta}$ a fixed point equation algorithm is
also presented.

\subsection{Derivation of the Estimating Equations for $\mathbf{\beta}$
\label{sec:equationsBeta}}

Hereafter, we are going to derive the expression of $T_{\mathbf{\beta}%
}(\mathbf{\beta},\mathbf{\gamma})=\sum_{n,j,l}\nabla_{\mathbf{\beta}}\tau
_{jl}(\mathbf{\beta},\mathbf{\gamma})$. For this aim, let $\dot{\mathbf{y}%
}_{i}^{jl}=\Sigma_{jl}^{\ast-1/2}\mathbf{y}_{i}^{jl}$ and $\dot{\mathbf{x}%
}_{i}^{jl}=\Sigma_{jl}^{\ast-1/2}\mathbf{x}_{i}^{jl}$ then
\begin{align*}
m_{i}^{jl}(\mathbf{\beta},\mathbf{\gamma})  &  =\left(  \mathbf{y}_{i}%
^{jl}-\mathbf{x}_{i}^{jl}\mathbf{\beta}\right)  ^{\top}\Sigma_{jl}^{\ast
-1}\left(  \mathbf{y}_{i}^{jl}-\mathbf{x}_{i}^{jl}\mathbf{\beta}\right) \\
&  =\left(  \dot{\mathbf{y}}_{i}^{jl}-\dot{\mathbf{x}}_{i}^{jl}\mathbf{\beta
}\right)  ^{\top}\left(  \dot{\mathbf{y}}_{i}^{jl}-\dot{\mathbf{x}}_{i}%
^{jl}\mathbf{\beta}\right)  .
\end{align*}
The derivative of the squared Mahalanobis distances is
\begin{align*}
\nabla_{\mathbf{\beta}}m_{i}^{jl}(\mathbf{\beta},\mathbf{\gamma})  &
=-2\dot{\mathbf{x}}_{i}^{jl\top}\left(  \dot{\mathbf{y}}_{i}^{jl}%
-\dot{\mathbf{x}}_{i}^{jl}\mathbf{\beta}\right) \\
&  =-2\left(  \dot{\mathbf{x}}_{i}^{jl\top}\dot{\mathbf{y}}_{i}^{jl}%
-\dot{\mathbf{x}}_{i}^{jl\top}\dot{\mathbf{x}}_{i}^{jl}\mathbf{\beta}\right)
\\
&  =-2\left(  \mathbf{x}_{i}^{jl\top}\Sigma_{jl}^{\ast-1}(\mathbf{\gamma
})\mathbf{y}_{i}^{jl}-\mathbf{x}_{i}^{jl\top}\Sigma_{jl}^{\ast-1}%
(\mathbf{\gamma})\mathbf{x}_{i}^{jl}\mathbf{\beta}\right)  .
\end{align*}
Let $W_{k}(x)=\rho_{k}^{\prime}(x)$ ($k=1,2$) and
\begin{equation*}
W_{k,i}^{jl}(\mathbf{\beta},\mathbf{\gamma})=W_{k}\left(  \frac{m_{i}%
^{jl}(\mathbf{\beta},\mathbf{\gamma})}{s_{jl}(\mathbf{\beta},\mathbf{\gamma}%
)}\right)
\end{equation*}
be a weight function. We compute the derivative of $s_{jl}(\mathbf{\beta
},\mathbf{\gamma})$ with respect to $\mathbf{\beta}$. We consider the
equality
\begin{equation*}
\frac{1}{n}\sum_{i=1}^{n}\rho_{1}\left(  \frac{m_{i}^{jl}(\mathbf{\beta
},\mathbf{\gamma})}{s_{jl}(\mathbf{\beta},\mathbf{\gamma})}\right)  =b
\end{equation*}
and we differentiate both sides
\begin{equation*}
\frac{1}{n}\sum_{i=1}^{n}\nabla_{\mathbf{\beta}}\rho_{1}\left(  \frac
{m_{i}^{jl}(\mathbf{\beta},\mathbf{\gamma})}{s_{jl}(\mathbf{\beta
},\mathbf{\gamma})}\right)  =0
\end{equation*}
which leads to the equation
\begin{equation*}
\frac{1}{n}\sum_{i=1}^{n}W_{1}\left(  \frac{m_{i}^{jl}(\mathbf{\beta
},\mathbf{\gamma} )}{s_{jl}(\mathbf{\beta},\mathbf{\gamma})}\right)
\frac{s_{jl}(\mathbf{\beta},\mathbf{\gamma})\nabla_{\mathbf{\beta}}m_{i}%
^{jl}(\mathbf{\beta},\mathbf{\gamma})-m_{i}^{jl}(\mathbf{\beta},\mathbf{\gamma
})\nabla_{\mathbf{\beta}}s_{jl}(\mathbf{\beta},\mathbf{\gamma})}{s_{jl}%
^{2}(\mathbf{\beta},\mathbf{\gamma})}=0,
\end{equation*}
and replacing the terms by the previous calculation, leads to the following
expression for $\nabla_{\mathbf{\beta}}s_{jl}(\mathbf{\beta},\mathbf{\gamma
})$
\begin{align*}
\nabla_{\mathbf{\beta}}s_{jl}(\mathbf{\beta},\mathbf{\gamma})  &
=\frac{-2\frac{1}{n}\sum_{i=1}^{n}W_{1,i}^{jl}(\mathbf{\beta},\mathbf{\gamma
})s_{jl}(\mathbf{\beta},\mathbf{\gamma})\left(  \dot{\mathbf{x}}_{i}^{jl\top
}\dot{\mathbf{y}}_{i}^{jl}-\dot{\mathbf{x}}_{i}^{jl\top}\dot{\mathbf{x}}%
_{i}^{jl}\mathbf{\beta}\right)  }{\frac{1}{n}\sum_{i=1}^{n}W_{1,i}%
^{jl}(\mathbf{\beta},\mathbf{\gamma})m_{i}^{jl}(\mathbf{\beta},\mathbf{\gamma
})}\\
&  =-2\frac{1}{n}\sum_{i=1}^{n}\dot{W}_{1,i}^{jl}(\mathbf{\beta}%
,\mathbf{\gamma})\left(  \dot{\mathbf{x}}_{i}^{jl\top}\dot{\mathbf{y}}%
_{i}^{jl}-\dot{\mathbf{x}}_{i}^{jl\top}\dot{\mathbf{x}}_{i}^{jl}\mathbf{\beta
}\right)  ,
\end{align*}
where
\begin{equation*}
\dot{W}_{1,i}^{jl}(\mathbf{\beta},\mathbf{\gamma})=\frac{W_{1,i}%
^{jl}(\mathbf{\beta},\mathbf{\gamma})s_{jl}(\mathbf{\beta},\mathbf{\gamma}%
)}{\frac{1}{n}\sum_{i=1}^{n}W_{1,i}^{jl}(\mathbf{\beta},\mathbf{\gamma}%
)m_{i}^{jl}(\mathbf{\beta},\mathbf{\gamma})}.
\end{equation*}
We are going to derive $\nabla_{\mathbf{\beta}}\tau_{jl}(\mathbf{\beta
},\mathbf{\gamma})$ to this aim we have
\begin{equation*}
\nabla_{\mathbf{\beta}}\tau_{jl}(\mathbf{\beta},\mathbf{\gamma})=\nabla
_{\mathbf{\beta}}s_{jl}(\mathbf{\beta},\mathbf{\gamma})\frac{1}{n}\sum
_{i=1}^{n}\rho_{2}\left(  \frac{m_{i}^{jl}(\mathbf{\beta},\mathbf{\gamma}%
)}{s_{jl}(\mathbf{\beta},\mathbf{\gamma})}\right)  +s_{jl}(\mathbf{\beta
},\mathbf{\gamma})\frac{1}{n}\sum_{i=1}^{n}\nabla_{\mathbf{\beta}}\rho
_{2}\left(  \frac{m_{i}^{jl}(\mathbf{\beta},\mathbf{\gamma})}{s_{jl}%
(\mathbf{\beta},\mathbf{\gamma})}\right)
\end{equation*}
and since
\begin{equation*}
\nabla_{\mathbf{\beta}}\rho_{2}\left(  \frac{m_{i}^{jl}(\mathbf{\beta
},\mathbf{\gamma})}{s_{jl}(\mathbf{\beta},\mathbf{\gamma})}\right)  =\left.
\rho_{2}^{\prime}(x)\right\vert _{x=m_{i}^{jl}(\mathbf{\beta},\mathbf{\gamma
})/s_{jl}(\mathbf{\beta},\mathbf{\gamma})}\ \frac{\nabla_{\mathbf{\beta}}%
m_{i}^{jl}(\mathbf{\beta},\mathbf{\gamma})s_{jl}(\mathbf{\beta},\mathbf{\gamma
})-\nabla_{\mathbf{\beta}}s_{jl}(\mathbf{\beta},\mathbf{\gamma})m_{i}%
^{jl}(\mathbf{\beta},\mathbf{\gamma})}{s_{jl}^{2}(\mathbf{\beta}%
,\mathbf{\gamma})}%
\end{equation*}
by further letting
\begin{align*}
A_{k,jl}  &  =\frac{1}{n}\sum_{i=1}^{n}\rho_{k}\left(  \frac{m_{i}%
^{jl}(\mathbf{\beta},\mathbf{\gamma})}{s_{jl}(\mathbf{\beta},\mathbf{\gamma}%
)}\right)  ,\\
B_{k,jl}  &  =\frac{1}{n}\sum_{i=1}^{n}W_{k}\left(  \frac{m_{i}^{jl}%
(\mathbf{\beta},\mathbf{\gamma})}{s_{jl}(\mathbf{\beta},\mathbf{\gamma}%
)}\right)  \frac{m_{i}^{jl}(\mathbf{\beta},\mathbf{\gamma})}{s_{jl}%
(\mathbf{\beta},\mathbf{\gamma})},
\end{align*}
we obtain
\begin{align*}
\nabla_{\mathbf{\beta}}\tau_{jl}(\mathbf{\beta},\mathbf{\gamma})  &
=-\frac{2}{n}\sum_{i=1}^{n}A_{2,jl}\dot{W}_{1,i}^{jl}\left(  \dot{\mathbf{x}%
}_{i}^{jl\top}\dot{\mathbf{y}}_{i}^{jl}-\dot{\mathbf{x}}_{i}^{jl\top}%
\dot{\mathbf{x}}_{i}^{jl}\mathbf{\beta}\right) \\
&  -\frac{2}{n}\sum_{i=1}^{n}W_{2,i}^{jl}\left(  \dot{\mathbf{x}}_{i}^{jl\top
}\dot{\mathbf{y}}_{i}^{jl}-\dot{\mathbf{x}}_{i}^{jl\top}\dot{\mathbf{x}}%
_{i}^{jl}\mathbf{\beta}\right) \\
&  +\frac{2}{n}\sum_{i=1}^{n}W_{2,i}^{jl}\left[  \frac{m_{i}^{jl}%
(\mathbf{\beta},\mathbf{\gamma})}{s_{jl}(\mathbf{\beta},\mathbf{\gamma})}%
\frac{1}{n}\sum_{k=1}^{n}\dot{W}_{1,k}^{jl}\left(  \dot{\mathbf{x}}%
_{k}^{jl\top}\dot{\mathbf{y}}_{k}^{jl}-\dot{\mathbf{x}}_{k}^{jl\top}%
\dot{\mathbf{x}}_{k}^{jl}\mathbf{\beta}\right)  \right] \\
&  =-\frac{2}{n}\sum_{i=1}^{n}\left(  (A_{2,jl}-B_{2,jl})\dot{W}_{1,i}%
^{jl}+W_{2,i}^{jl}\right)  \left(  \dot{\mathbf{x}}_{i}^{jl\top}%
\dot{\mathbf{y}}_{i}^{jl}-\dot{\mathbf{x}}_{i}^{jl\top}\dot{\mathbf{x}}%
_{i}^{jl}\mathbf{\beta}\right)  .
\end{align*}
Hence differentiating (\ref{goal-tau}) with respect to $\mathbf{\beta}$ leads
to
\begin{equation*}
T_{\mathbf{\beta}}(\mathbf{\beta},\mathbf{\gamma})=-\frac{2}{n}\sum_{i=1}%
^{n}\sum_{j=1}^{p-1}\sum_{l=j+1}^{p}\left(  (A_{2,jl}-B_{2,jl})\dot{W}%
_{1,i}^{jl}+W_{2,i}^{jl}\right)  \left(  \dot{\mathbf{x}}_{i}^{jl\top}%
\dot{\mathbf{y}}_{i}^{jl}-\dot{\mathbf{x}}_{i}^{jl\top}\dot{\mathbf{x}}%
_{i}^{jl}\mathbf{\beta}\right)  =\mathbf{0},
\end{equation*}
and finally, by letting $\tilde{W}_{i}^{jl}=(A_{2,jl}-B_{2,jl})\dot{W}%
_{1,i}^{jl}+W_{2,i}^{jl}$ a fixed point equation for $\mathbf{\beta}$ is
\begin{equation}
\label{equ:Tbeta}\mathbf{\beta}=\left[  \sum_{i=1}^{n}\sum_{j=1}^{p-1}%
\sum_{l=j+1}^{p}\tilde{W}_{i}^{jl}(\mathbf{\beta},\mathbf{\gamma})\left(
\dot{\mathbf{x}}_{i}^{jl\top}\dot{\mathbf{x}}_{i}^{jl}\right)  \right]
^{-1}\sum_{i=1}^{n}\sum_{j=1}^{p-1}\sum_{l=j+1}^{p}\tilde{W}_{i}%
^{jl}(\mathbf{\beta},\mathbf{\gamma})\left(  \dot{\mathbf{x}}_{i}^{jl\top}%
\dot{\mathbf{y}}_{i}^{jl}\right)  .
\end{equation}

\subsection{Derivation of the Estimating Equations for $\mathbf{\gamma}$
\label{sec:equationsGamma}}

Let the residual $\mathbf{r}_{i}^{jl}$ be as follows
\begin{equation*}
\mathbf{r}_{i}^{jl}=\mathbf{y}_{i}^{jl}-\mathbf{x}_{i}^{jl}\mathbf{\beta}.
\end{equation*}
In order to obtain the estimating equations for $\mathbf{\gamma}$ we have to
differentiate the function (\ref{goal-tau}) with respect to $\mathbf{\gamma}$.
Let us write
\begin{equation*}
m_{i}^{jl}(\mathbf{\beta},\mathbf{\gamma})=\mathbf{r}_{i}^{jl\top}\Sigma
_{jl}^{\ast\ -1}(\mathbf{\gamma})\mathbf{r}_{i}^{jl},
\end{equation*}
and therefore
\begin{equation*}
\frac{\partial}{\partial\gamma_{r}}m_{i}^{jl}(\mathbf{\beta},\mathbf{\gamma
})=\mathbf{r}_{i}^{jl\top}\frac{\partial}{\partial\gamma_{r}}\Sigma_{jl}%
^{\ast\ -1}(\mathbf{\gamma})\mathbf{r}_{i}^{jl}.
\end{equation*}
Call $\sigma_{jl}$ the $(j,l)$ element of $\Sigma(1,\mathbf{\gamma})$ and
$v_{r,jl}$ the $(j,l)$ element of $V_{r}$ ($r=1,\ldots,J$). We are going to
assume without loss of generality that $\eta=1$. We can write
\begin{equation}
\Sigma_{jl}^{\ast\ -1}=|\Sigma_{jl}|^{-1/2}\left(
\begin{array}
[c]{ll}%
\sigma_{ll} & -\sigma_{jl}\\
-\sigma_{jl} & \sigma_{jj}%
\end{array}
\right)  . \label{sigmastarinv}%
\end{equation}
Since $\sigma_{jl}=\delta_{jl}+\sum_{r=1}^{J}\gamma_{r}v_{r,jl}$ where
$\delta_{jl}=1$ if $j=l$ and $0$ otherwise, and $\frac{\partial}%
{\partial\gamma_{r}}\sigma_{jl}=v_{r,jl}$ we have
\begin{align*}
\frac{\partial}{\partial\gamma_{r}}|\mathbf{\Sigma}_{jl}|  &  =(v_{r,ll}%
\sigma_{jj}+\sigma_{ll}v_{r,jj}-2v_{r,jl}\sigma_{jl})\\
&  =2c_{r,jl},
\end{align*}
where
\begin{equation*}
c_{r,jl}=\frac{1}{2}v_{r,jj}\sigma_{ll}+\frac{1}{2}v_{r,ll}\sigma
_{jj}-v_{r,jl}\sigma_{jl}.\label{Crjl}%
\end{equation*}
Moreover,
\begin{equation*}
\left(
\begin{array}
[c]{ll}%
\partial\sigma_{ll}/\partial\gamma_{r} & -\partial\sigma_{jl}/\partial
\gamma_{r}\\
-\partial\sigma_{jl}/\partial\gamma_{r} & \partial\sigma_{jj}/\partial
\gamma_{r}%
\end{array}
\right)  =\left(
\begin{array}
[c]{ll}%
v_{r,ll} & -v_{r,jl}\\
-v_{r,jl} & v_{r,jj}%
\end{array}
\right)  =|V_{r,jl}|V_{r,jl}^{-1}.
\end{equation*}
Then noting that $\left(
\begin{array}
[c]{ll}%
\sigma_{ll} & -\sigma_{jl}\\
-\sigma_{jl} & \sigma_{jj}%
\end{array}
\right)  =|\Sigma_{jl}|\ \Sigma_{jl}^{-1}$ differentiating (\ref{sigmastarinv}%
) is
\begin{align*}
\frac{\partial}{\partial\gamma_{r}}\Sigma_{jl}^{\ast-1}  &  =\left(
\begin{array}
[c]{ll}%
\sigma_{ll} & -\sigma_{jl}\\
-\sigma_{jl} & \sigma_{jj}%
\end{array}
\right)  \frac{\partial}{\partial\gamma_{r}}|\Sigma_{jl}|^{-1/2}+|\Sigma
_{jl}|^{-1/2}\ \left(
\begin{array}
[c]{ll}%
\partial\sigma_{ll}/\partial\gamma_{r} & -\partial\sigma_{jl}/\partial
\gamma_{r}\\
-\partial\sigma_{jl}/\partial\gamma_{r} & \partial\sigma_{jj}/\partial
\gamma_{r}%
\end{array}
\right) \\
&  =-|\Sigma_{jl}|^{-3/2}\left(
\begin{array}
[c]{ll}%
\sigma_{ll} & -\sigma_{jl}\\
-\sigma_{jl} & \sigma_{jj}%
\end{array}
\right)  c_{r,jl}+|\Sigma_{jl}|^{-1/2}|V_{r,jl}|V_{r,jl}^{-1}\\
&  =-|\Sigma_{jl}|^{-3/2}|\Sigma_{jl}|\Sigma_{jl}^{-1}c_{r,jl}+|\Sigma
_{jl}|^{-1/2}|V_{r,jl}|V_{r,jl}^{-1}\\
&  =|\Sigma_{jl}|^{-1/2}\left(  +|V_{r,jl}|V_{r,jl}^{-1}-c_{r,jl}\Sigma
_{jl}^{-1}\right)  .
\end{align*}
Hence,
\begin{equation*}
\frac{\partial}{\partial\gamma_{r}}m_{i}^{jl}(\mathbf{\beta},\mathbf{\gamma
})=|\Sigma_{jl}|^{-1/2}\mathbf{r}_{i}^{jl\top}[|V_{r,jl}|V_{r,jl}%
^{-1}-c_{r,jl}\Sigma^{-1}]\mathbf{r}_{i}^{jl}.
\end{equation*}
We now compute the derivative of $s_{jl}(\mathbf{\beta},\mathbf{\gamma})$ with
respect to $\gamma_{r}$. We consider the equality
\begin{equation*}
\frac{1}{n}\sum_{i=1}^{n}\rho_{1}\left(  \frac{m_{i}^{jl}(\mathbf{\beta
},\mathbf{\gamma})}{s_{jl}(\mathbf{\beta},\mathbf{\gamma})}\right)  =b
\end{equation*}
and we differentiate both sides
\begin{equation*}
\frac{1}{n}\sum_{i=1}^{n}\frac{\partial}{\partial\gamma_{r}}\rho_{1}\left(
\frac{m_{i}^{jl}(\mathbf{\beta},\mathbf{\gamma})}{s_{jl}(\mathbf{\beta
},\mathbf{\gamma})}\right)  =0,\qquad1\leq r\leq J,
\end{equation*}
which leads to the equation
\begin{equation*}
\frac{1}{n}\sum_{i=1}^{n}W_{1,i}^{jl}(\mathbf{\beta},\mathbf{\gamma}%
)\frac{\frac{\partial}{\partial\gamma_{r}}m_{i}^{jl}(\mathbf{\beta
},\mathbf{\gamma})}{s_{jl}(\mathbf{\beta},\mathbf{\gamma})}-\frac{1}{n}%
\sum_{i=1}^{n}W_{1,i}^{jl}(\mathbf{\beta},\mathbf{\gamma})\frac{m_{i}%
^{jl}(\mathbf{\beta},\mathbf{\gamma})}{s_{jl}^{2}(\mathbf{\beta}%
,\mathbf{\gamma})}\frac{\partial}{\partial\gamma_{r}}s_{jl}(\mathbf{\beta
},\mathbf{\gamma})=0,
\end{equation*}
and replacing the terms by the previous calculation, leads to the following
expression for $\frac{\partial}{\partial\gamma_{r}}s_{jl}(\mathbf{\beta
},\mathbf{\gamma})$
\begin{align*}
\frac{\partial}{\partial\gamma_{r}}s_{jl}(\mathbf{\beta},\mathbf{\gamma})  &
=\frac{s_{jl}(\mathbf{\beta},\mathbf{\gamma})\frac{1}{n}\sum_{i=1}^{n}%
W_{1,i}^{jl}(\mathbf{\beta},\mathbf{\gamma})\ \frac{\partial}{\partial
\gamma_{r}}m_{i}^{jl}(\mathbf{\beta},\mathbf{\gamma})}{\frac{1}{n}\sum
_{i=1}^{n}W_{1,i}^{jl}(\mathbf{\beta},\mathbf{\gamma})m_{i}^{jl}%
(\mathbf{\beta},\mathbf{\gamma})}\\
&  =\frac{1}{n}\sum_{i=1}^{n}\dot{W}_{1,i}^{jl}(\mathbf{\beta},\mathbf{\gamma
})\frac{\partial}{\partial\gamma_{r}}m_{i}^{jl}(\mathbf{\beta},\mathbf{\gamma
})\\
&  =\frac{1}{n}\sum_{i=1}^{n}\dot{W}_{1,i}^{jl}(\mathbf{\beta},\mathbf{\gamma
})|\Sigma_{jl}|^{-1/2}\mathbf{r}_{i}^{jl\top}\left(  |V_{r,jl}|V_{r,jl}%
^{-1}-c_{r,jl}\Sigma_{jl}^{-1}\right)  \mathbf{r}_{i}^{jl}.
\end{align*}
Going back to the derivative $\frac{\partial}{\partial\gamma_{r}}\tau_{jl}$ we
obtain
\begin{align*}
\frac{\partial}{\partial\gamma_{r}}\tau_{jl}  &  =\frac{1}{n}\sum_{i=1}%
^{n}\rho_{2}\left(  \frac{m_{i}^{jl}(\mathbf{\beta},\mathbf{\gamma})}%
{s_{jl}(\mathbf{\beta},\mathbf{\gamma})}\right)  \frac{\partial}%
{\partial\gamma_{r}}s_{jl}(\mathbf{\beta},\mathbf{\gamma})\\
&  +\frac{1}{n}\sum_{i=1}^{n}W_{2,i}\left(  \frac{m_{i}^{jl}(\mathbf{\beta
},\mathbf{\gamma})}{s_{jl}(\mathbf{\beta},\mathbf{\gamma})}\right)  \left[
\frac{\partial}{\partial\gamma_{r}}m_{i}^{jl}(\mathbf{\beta},\mathbf{\gamma
})-\frac{m_{i}^{jl}(\mathbf{\beta},\mathbf{\gamma})}{s_{jl}(\mathbf{\beta
},\mathbf{\gamma})}\frac{\partial}{\partial\gamma_{r}}s_{jl}(\mathbf{\beta
},\mathbf{\gamma})\right] \\
&  =\frac{1}{n}\sum_{i=1}^{n}\left[  \left(  A_{2,jl}-B_{2,jl}\right)  \dot
{W}_{1,i}^{jl}(\mathbf{\beta},\mathbf{\gamma})+W_{2,i}^{jl}(\mathbf{\beta
},\mathbf{\gamma})\right]  \frac{\partial}{\partial\gamma_{r}}m_{i}%
^{jl}(\mathbf{\beta},\mathbf{\gamma}).
\end{align*}
which leads to the estimating equation ($r=1,\ldots,J$)
\begin{equation}
T_{\gamma_{r}}(\mathbf{\beta},\mathbf{\gamma})=\frac{1}{n}\sum_{i=1}^{n}%
\sum_{j=1}^{p-1}\sum_{l=j+1}^{p}\frac{\tilde{W}_{i}^{jl}(\mathbf{\beta
},\mathbf{\gamma})}{\left\vert \Sigma_{jl}{}^{-1/2}(1,\mathbf{\gamma
})\right\vert ^{1/2}}\ \mathbf{r}_{i}^{jl\top}(\mathbf{\beta})\left(
|V_{r,jl}|V_{r,jl}^{-1}-c_{r,jl}\Sigma_{jl}^{-1}(1,\mathbf{\gamma})\right)
\mathbf{r}_{i}^{jl}(\mathbf{\beta})=0. \label{equ:Tgamma}%
\end{equation}

\section{Breakdown Point \label{app:breakdown1}}

To prove Theorems \ref{teo:classicbp} and \ref{teo:independentbp} that give
lower bounds for the FSBDPCC and FSBDPIC of the composite $\tau$-estimators
respectively, we need the following Lemmas.

\begin{lemma}
\label{lemmadelta} Consider a sample $\mathbf{T}$ $=(\mathbf{t}_{1}%
,\ldots,\mathbf{t}_{n})$. Let $h=h(\mathbf{T}$) and $h^{\ast}=h^{\ast
}(\mathbf{T}$) be defined by (\ref{equ:h}) and (\ref{equ:hstar}) respectively.
Define for all ($j,l)$%
\begin{equation*}
\delta_{jl}=\inf_{\Vert\mathbf{b}\Vert=1,\mathbf{b}\in\mathbb{R}^{k},}%
\inf_{1\leq i_{1}<\cdots<i_{h}<i_{{h}+1}\leq n}\max\{\Vert\mathbf{x}_{i_{1}%
}^{jl}\mathbf{b}\Vert,\ldots,\Vert\mathbf{x}_{i_{h}}^{jl}\mathbf{b}\Vert
,\Vert\mathbf{x}_{i_{{h}+1}}^{jl}\mathbf{b}\Vert\}\label{deltajl}%
\end{equation*}
and
\begin{align}
\delta_{jl}^{\ast}=\inf_{\Vert\mathbf{u}\Vert=1,\mathbf{u}\in\mathbb{R}%
^{p},\mathbf{b}\in\mathbb{R}^{k},}\inf_{1\leq i_{1}<\cdots<i_{h^{\ast}%
}<i_{{h^{\ast}}+1}\leq n}  &  \max\{\Vert\mathbf{u}^{\top}(\mathbf{y}_{i_{1}%
}^{jl}-\mathbf{x}_{i_{1}}^{jl}\mathbf{b})\Vert,\ldots,\nonumber\\
&  \Vert\mathbf{u}^{\top}(\mathbf{y}_{i_{h^{\ast}}}^{jl}-\mathbf{x}%
_{i_{h^{\ast}}}^{jl}\mathbf{b})\Vert,\Vert\mathbf{u}^{\top}(\mathbf{y}%
_{i_{h^{\ast}+1}}^{jl}-\mathbf{x}_{i_{h^{\ast}+1}}^{jl}\mathbf{b})\Vert\}.
\label{deltastarjl}%
\end{align}
Then, $\delta=\min_{jl}\delta_{jl}>0$ and $\delta^{\ast}=\min_{jl}\delta
_{jl}^{\ast}>0$.
\end{lemma}

\begin{proof}
The proof follows from the definition of $h(\mathbf{T}$) and $h^{\ast
}(\mathbf{T}$) in (\ref{equ:h}) and (\ref{equ:hstar}) respectively.
\end{proof}

\begin{lemma}
\label{lemma:bdpgamma} Consider the same assumptions as in Theorem
(\ref{teo:classicbp}), a sample $\mathbf{T}$ $=(\mathbf{t}_{1},\ldots
,\mathbf{t}_{n})$ and let $m<\min((1-b)n-f(\mathbf{T}),bn)$. Then, if
$\mathbf{\check{T}}\in\mathcal{T}_{m}^{(\text{C})}$ and $\widehat
{\mathbf{\gamma}}(\mathbf{\check{T}}$) is the $\tau$-estimator of
$\mathbf{\gamma}$ for the sample $\mathbf{\check{T}}$, there exists $K>0$ that
for all couples $(j,l)$ $(1\leq j<l\leq p)$, the two eigenvalues of
$\Sigma_{jl}^{\ast}(\widehat{\mathbf{\gamma}}(\mathbf{\check{T}}))$,
$\omega_{jl}^{-}(\mathbf{\check{T}})\leq\omega_{jl}^{+}(\mathbf{\check{T}})$
are such that
\begin{equation*}
1\leq\frac{\omega_{jl}^{+}(\mathbf{\check{T}})}{\omega_{jl}^{-}(\mathbf{\check
{T}})}\leq K.
\end{equation*}

\end{lemma}

\begin{proof}
Suppose that lemma is not true for the couple $(j,l)$. Then there exists a
sequence $\{\mathbf{\check{T}}_{N}\}_{1\leq N<\infty}$ such that $\omega
_{jl}^{+}(\mathbf{\check{T}}_{N})/\omega_{jl}^{-}(\mathbf{\check{T}}%
_{N})\rightarrow\infty$. Since $|\Sigma_{jl}^{\ast}(\widehat{\mathbf{\gamma}%
}(\mathbf{\check{T}}_{N}))|=1$ this is equivalent to $\omega_{jl}%
^{-}(\mathbf{\check{T}}_{N})\rightarrow0$. We are going to show that, for this
sequence,
\begin{equation}
\lim_{N\rightarrow\infty}\tau_{jl}(\widehat{\mathbf{\beta}}(\mathbf{\check{T}%
}_{N}),\widehat{\mathbf{\gamma}}(\mathbf{\check{T}}_{N}))=\infty.
\label{equ:snf}%
\end{equation}
Let $U_{N}$ be an orthogonal matrix of eigenvectors of $\Sigma_{jl}^{\ast
-1}(\widehat{\mathbf{\gamma}}(\mathbf{\check{T}}_{N}))$ and $\Lambda_{N}$ be
the diagonal matrix with the corresponding eigenvalues, i.e., $\Lambda
_{N}=\operatorname{diag}(\lambda_{jlN}^{+}=1/\omega_{jlN}^{-},\lambda
_{jlN}^{-}=1/\omega_{jlN}^{+})$ and let
\begin{equation*}
\mathbf{g}_{Ni}=\left(
\begin{array}
[c]{l}%
g_{Ni1}\\
g_{Ni2}%
\end{array}
\right)  =\left(  U_{N}^{\top}\mathbf{y}_{Ni}^{jl}-U_{N}^{\top}\mathbf{x}%
_{Ni}^{jl}\widehat{\mathbf{\beta}}(\mathbf{\check{T}}_{N})\right)  ,
\end{equation*}
then, calling $\mathbf{u}_{N1}$ the first column of $U_{N}$ we get
\begin{align}
m_{Ni}^{jl}(\widehat{\mathbf{\beta}}(\mathbf{\check{T}}_{N}),\widehat
{\mathbf{\gamma}}(\mathbf{\check{T}}_{N}))  &  =\left(  \mathbf{y}_{Ni}%
^{jl}-\mathbf{x}_{Ni}^{jl}\widehat{\mathbf{\beta}}(\mathbf{\check{T}}%
_{N})\right)  ^{\top}U_{N}\Lambda_{N}U_{N}^{\top}\left(  \mathbf{y}_{Ni}%
^{jl}-\mathbf{x}_{Ni}^{jl}\widehat{\mathbf{\beta}}(\mathbf{\check{T}}%
_{N})\right) \nonumber\\
&  =\left(  U_{N}^{\top}\mathbf{y}_{Ni}^{jl}-U_{N}^{\top}\mathbf{x}_{Ni}%
^{jl}\widehat{\mathbf{\beta}}(\mathbf{\check{T}}_{N})\right)  ^{\top}%
\Lambda_{N}\left(  U_{N}^{\top}\mathbf{y}_{Ni}^{jl}-U_{N}^{\top}%
\mathbf{x}_{Ni}^{jl}\widehat{\mathbf{\beta}}(\mathbf{\check{T}}_{N})\right)
\nonumber\\
&  \geq\left(  \mathbf{u}_{N1}^{T}\mathbf{y}_{Ni}^{jl}-\mathbf{u}_{N1}%
^{T}\mathbf{x}_{Ni}^{jlT}\widehat{\mathbf{\beta}}(\mathbf{\check{T}}%
_{N})\right)  ^{2}\lambda_{jlN}^{+}\nonumber\\
&  =g_{Ni1}^{2}\lambda_{jlN}^{+}\nonumber\\
&  =\frac{g_{Ni1}^{2}}{\omega_{jlN}^{-}}. \label{equ:meta}%
\end{align}
By Lemma \ref{lemmadelta} there exists $\delta^{\ast}>0$ such that for at
least $n-f(\mathbf{T})$ observations from $\mathbf{T}$ we have
\begin{equation*}
\inf_{\Vert\mathbf{u}\Vert=1,\mathbf{b}}\Vert\mathbf{u}^{\top}(\mathbf{y}%
_{i}^{jl}-\mathbf{x}_{i}^{jl}\mathbf{b})\Vert>\delta^{\ast}%
\end{equation*}
and we can find in any sample $\mathbf{\check{T}}_{N}$ more than
$[bn]+f(\mathbf{T})$ observations from the original sample and therefore
$[bn]+1$ indexes $i_{1},\ldots,i_{q},\ldots,i_{[bn]+1}$ such that
\begin{equation*}
g_{N{i_{q}}1}^{2}\geq\delta^{\ast},\text{ }1\leq q\leq\lbrack bn]+1.
\end{equation*}
Then according to equation (\ref{equ:meta}) there are more than $[bn]+1$
observations such that
\begin{equation*}
m_{N{i_{q}}}^{jl}(\widehat{\mathbf{\beta}}(\mathbf{\check{T}}_{N}),\hat
{\gamma}(\mathbf{\check{T}}_{N}))\rightarrow\infty.
\end{equation*}
Using Lemma A.3 of \citet{garciaben_martinez_yohai_2006} (see also
\citet{yohai_zamar_1986}) this implies (\ref{equ:snf}). On the other hand if
we put $\widehat{\mathbf{\beta}}_{N}=\mathbf{0}$ and $\widehat{\mathbf{\gamma
}}_{N}=\mathbf{0}$ we will have, for all pair $(j,l)$ more than $[bn]+1$
observations such that the corresponding squared Mahalanobis distances are
uniformly bounded on $N$ and therefore by Lemma A.1 of
\citet{garciaben_martinez_yohai_2006} (see also \citet{yohai_zamar_1986}) all
the $\tau_{jl}(\mathbf{0},\mathbf{0})$ will be uniformly bounded and therefore
$T(\mathbf{0},\mathbf{0})$ will be finite. This contradicts the definition of
the composite $\tau$-estimator for $\mathbf{\beta}$ and $\mathbf{\gamma}$.
\end{proof}

\begin{lemma}
\label{lemma:bdpbeta} Consider the same assumptions as in Theorem
(\ref{teo:classicbp}). Then for any $m<\min((1-b)n-f(\mathbf{T}),bn)$ we have
$B_{m}^{(\text{C})}(\mathbf{T},\widehat{\mathbf{\beta}})<\infty$.
\end{lemma}

\begin{proof}
Assume that there exists a sequence $\{\mathbf{\check{T}}_{N}\}_{N}$ with
$\mathbf{\check{T}}_{N}\in\mathcal{T}_{m}^{(\text{C})}$ such that
$\Vert\widehat{\mathbf{\beta}}(\mathbf{\check{T}}_{N})\Vert\rightarrow\infty$
as $N\rightarrow\infty$. We can assume without loss of generality that
$\widehat{\mathbf{\beta}}(\mathbf{\check{T}}_{N})/\Vert\widehat{\mathbf{\beta
}}(\mathbf{\check{T}}_{N})\Vert\rightarrow\mathbf{c}_{N}$. Then, we will show
that there exists a pair $(j,l)$ such that $\tau_{jl}(\widehat{\mathbf{\beta}%
}(\mathbf{\check{T}}_{N}),\widehat{\mathbf{\gamma}}(\mathbf{\check{T}}%
_{N}))\rightarrow\infty$ as $N\rightarrow\infty.$ Let $U_{N}$ be an orthogonal
matrix of eigenvectors of $\Sigma_{jl}^{\ast-1}(\widehat{\mathbf{\gamma}%
}(\mathbf{\check{T}}_{N}))$ and $\Lambda_{N}$ be the diagonal matrix with the
corresponding eigenvalues as in the proof of Lemma \ref{lemma:bdpgamma},
hence
\begin{equation*}
m_{Ni}^{jl}(\widehat{\mathbf{\beta}}(\mathbf{\check{T}}_{N}),\widehat
{\mathbf{\gamma}}(\mathbf{\check{T}}_{N}))=\mathbf{g}_{Ni}^{\top}\Lambda
_{N}\mathbf{g}_{Ni}.
\end{equation*}
and
\begin{equation*}
\mathbf{g}_{Ni}=\left(  U_{Ni}^{\top}\mathbf{y}_{Ni}^{jl}-U_{Ni}^{\top
}\mathbf{x}_{Ni}^{jl}\mathbf{c}_{N}\Vert\widehat{\mathbf{\beta}}%
(\mathbf{\check{T}}_{N})\Vert\right)  .
\end{equation*}
By Lemma \ref{lemmadelta} there exists $\delta>0$ such that for at least
$n-f(\mathbf{T}$) observations from $\mathbf{T}$ we have
\begin{equation}
\inf_{\Vert\mathbf{b}\Vert=1}\Vert(\mathbf{x}_{i}^{jl}\mathbf{b})\Vert
>\delta\label{del}%
\end{equation}
and we can find in any sample $\mathbf{\check{T}}_{N}$ more than
$[bn]+f(\mathbf{T})$ observations from the original sample. Therefore we can
find $[bn]+1$ indexes $i_{1},\ldots,i_{q},\ldots,i_{[bn]}$ such that
\begin{equation*}
\left\vert \left\vert \mathbf{x}_{Ni_{q}}^{jl}\frac{\widehat{\mathbf{\beta}%
}(\mathbf{\check{T}}_{N})}{\Vert\widehat{\mathbf{\beta}}(\mathbf{\check{T}%
}_{N})\Vert}\right\vert \right\vert \geq\delta,\text{ }1\leq q\leq\lbrack
bn]+1,
\end{equation*}
then,
\begin{equation*}
\lim_{N\rightarrow\infty}g_{Ni_{q}k}=\infty,\text{ }1\leq q\leq\lbrack
bn]+1,\text{ }k=1,2.
\end{equation*}
According to the Lemma \ref{lemma:bdpgamma} the diagonal elements of the
matrix $\Lambda_{N}$ are greater than some positive constant $K$, this implies
that for any vector $\mathbf{a}$
\begin{equation}
\mathbf{a}^{\top}\Lambda_{N}\mathbf{a}>K\Vert a\Vert^{2}, \label{equ:Lambda}%
\end{equation}
which leads for $1\leq q\leq\lbrack bn]+1$ to
\begin{align*}
\lim_{N\rightarrow\infty}m_{Ni_{q}}^{jl}(\widehat{\mathbf{\beta}%
}(\mathbf{\check{T}}_{N}),\widehat{\mathbf{\gamma}}(\mathbf{\check{T}}_{N}))
&  =\lim_{N\rightarrow\infty}\mathbf{g}_{Ni_{q}}^{T}\Lambda_{N}\mathbf{g}%
_{Ni_{q}}\\
&  \geq K\lim_{N\rightarrow\infty}\Vert\mathbf{g}_{Ni_{q}}\Vert^{2}\\
&  =\infty.
\end{align*}
Then, we have more than $[bn]+1$ elements $m_{Ni_{q}}^{jl}(\widehat
{\mathbf{\beta}}(\mathbf{\check{T}}_{N}),\widehat{\mathbf{\gamma}%
}(\mathbf{\check{T}}_{N}))$ going to infinity. Using Lemma A.3 of
\citet{garciaben_martinez_yohai_2006} (see also \citet{yohai_zamar_1986}) this
implies $\tau_{jl}(\widehat{\mathbf{\beta}}(\mathbf{\check{T}}_{N}%
),\widehat{\mathbf{\gamma}}(\mathbf{\check{T}}_{N}))\rightarrow\infty$ as
$N\rightarrow\infty$. Then $T(\widehat{\mathbf{\beta}}(\mathbf{\check{T}}%
_{N}),\widehat{\mathbf{\gamma}}(\mathbf{\check{T}}_{N}))$ tends to infinity
too. On the other hand if we put $\widehat{\mathbf{\beta}}_{N}=\mathbf{0}$ and
$\widehat{\mathbf{\gamma}}_{N}=\mathbf{0}$ we will have, for all pair $(j,l)$
more than $[bn]$ observations such that the corresponding squared Mahalanobis
distances are uniformly bounded on $N$ and therefore by Lemma A.1 of
\citet{garciaben_martinez_yohai_2006} (see also \citet{yohai_zamar_1986}) all
the $\tau_{jl}(\mathbf{0},\mathbf{0})$ will be uniformly bounded and therefore
$T(\mathbf{0},\mathbf{0})$ will be finite. This contradicts the definition of
the composite $\tau$-estimator for $\mathbf{\beta}$ and $\mathbf{\gamma}$.
\end{proof}

\begin{lemma}
\label{lemma:bdpeta} Consider the same assumptions as in Theorem
(\ref{teo:classicbp}). Then, for any $m<\min((1-b)n-f(\mathbf{T}),bn)$ we have
$B_{m}^{-(\text{C})}(\mathbf{T},\widehat{\mathbf{\upsilon}})>0$ and
$B_{m}^{+(\text{C})}(\mathbf{T},\widehat{\mathbf{\upsilon}})<\infty$.
\end{lemma}

\begin{proof}
Recall that $\mathbf{\upsilon}=(\eta,\eta\mathbf{\gamma})$ and that
$\Sigma(\eta,\mathbf{\gamma})=\eta\Sigma(1,\mathbf{\gamma})$ and therefore
from equation (\ref{equ:eta0}), $\eta$ is the solution of
\begin{equation}
\frac{2}{p(p-1)n}\sum_{i=1}^{n}\sum_{j=1}^{p-1}\sum_{l=j+1}^{p}\rho_{1}\left(
\frac{(\mathbf{y}_{i}^{jl}-\mathbf{x}_{i}^{jl}\widehat{\mathbf{\beta}})^{\top
}\Sigma_{jl}(\eta,\widehat{\mathbf{\gamma}})^{-1}(\mathbf{y}_{i}%
^{jl}-\mathbf{x}_{i}^{jl}\widehat{\mathbf{\beta}})}{s_{0}}\right)  =b,
\label{etaso}%
\end{equation}

where $0<s_{0}<\infty$ is defined as follows
\begin{equation*}
\mathbb{E}\left(  \rho\left(  \frac{v}{s_{0}}\right)  \right)  =b,\text{
}v\sim\chi_{2}^{2}.
\end{equation*}
Assume that there exists a sequence $\{\mathbf{\check{T}}_{N}\}_{N}$ with
$\mathbf{\check{T}}_{N}\in\mathcal{T}_{m}$, such that $\Vert\widehat
{\mathbf{\upsilon}}(\mathbf{\check{T}}_{n})\Vert\rightarrow0$ as
$N\rightarrow\infty$. This implies that all the eigenvalues of the
matrices $\Sigma_{jl}(\widehat{\mathbf{\upsilon}}(\mathbf{\check{T}}_{n}))$
converge to zero. Then all the eigenvalues of the matrices $\Sigma
_{jl}(\widehat{\mathbf{\upsilon}}(\mathbf{\check{T}}_{n}))^{-1}$ converge to
infinity. Let $U_{Njl}$ and $\Lambda_{Njl}=\operatorname{diag}(\lambda
_{Njl}^{+},\lambda_{Njl}^{-})$ be the eigenvectors and eigenvalues of these
matrices and let
\begin{equation*}
\mathbf{g}_{Ni}=\left(
\begin{array}
[c]{l}%
g_{Ni1}\\
g_{Ni2}%
\end{array}
\right)  =\left(  U_{Njl}^{\top}\mathbf{y}_{Ni}^{jl}-U_{Njl}^{\top}%
\mathbf{x}_{Ni}^{jl}\widehat{\mathbf{\beta}}(\mathbf{\check{T}}_{N})\right)
,
\end{equation*}
as in proof of Lemma \ref{lemma:bdpgamma}. As shown in Lemma \ref{lemmadelta}
there exists a $\delta^{\ast}>0$ such that for at least $n-f(\mathbf{T}$)
observations from $\mathbf{T}$ we have
\begin{equation*}
\inf_{\Vert\mathbf{u}\Vert=1,\mathbf{b}}\Vert\mathbf{u}^{\top}(\mathbf{y}%
_{i}^{jl}-\mathbf{x}_{i}^{jl}\mathbf{b})\Vert>\delta^{\ast},
\end{equation*}

and we can find in any sample $\mathbf{\check{T}}_{N}$ more than
$[bn]+f(\mathbf{T}$) observations from the original sample and therefore there
are $[bn]+1$ indexes $i_{1},\ldots,i_{q},\ldots,i_{[bn]+1}$ such that
\begin{equation*}
g_{N{i_{q}}1}^{2}\geq\delta^{\ast},\text{ }1\leq q\leq\lbrack bn]+1.
\end{equation*}
Then, for $1\leq q\leq\lbrack bn]+1$
\begin{align*}
(\mathbf{y}_{i_{q}}^{jl}-\mathbf{x}_{i_{q}}^{jl}\widehat{\mathbf{\beta}%
})^{\top}\Sigma_{jl}(\eta,\widehat{\mathbf{\gamma}})^{-1}(\mathbf{y}_{i_{q}%
}^{jl}-\mathbf{x}_{i_{q}}^{jl}\widehat{\mathbf{\beta}})  &  =\mathbf{g}%
_{Ni_{q}}^{\top}\Lambda_{Njl}\mathbf{g}_{Ni_{q}}\\
&  \geq g_{Ni_{q}1}^{2}\lambda_{Njl}^{+}\\
&  \geq\delta^{\ast}\lambda_{Njl}^{+}%
\end{align*}
and therefore we have%
\begin{equation*}
\lim_{N\rightarrow\infty}(\mathbf{y}_{i_{q}}^{jl}-\mathbf{x}_{i_{q}}%
^{jl}\widehat{\mathbf{\beta}})^{\top}\Sigma_{jl}(\eta,\widehat{\mathbf{\gamma
}})^{-1}(\mathbf{y}_{i_{q}}^{jl}-\mathbf{x}_{i_{q}}^{jl}\widehat
{\mathbf{\beta}})=\infty
\end{equation*}
for all pair $(j,l)$ and $1\leq q\leq\lbrack bn]+1.$ Then, the fraction of
squared Mahalanobis distances that goes to infinity in the left hand side of
equation (\ref{equ:eta0}) is going to be larger than $[bn].$ Then, according
to the Lemma A.3 of \citet{garciaben_martinez_yohai_2006} (see also
\citet{yohai_zamar_1986}), this implies that the scale should go to infinity.
This contradict the fact that according to (\ref{etaso}) this scale is always
$s_{0}$. Suppose that we assume that there exists a sequence $\mathbf{\check
{T}}_{N}$, $N\geq1$ with $\mathbf{\check{T}}_{N}\in\mathcal{T}_{m}$, such that
$\Vert\widehat{\mathbf{\upsilon}}(\mathbf{\check{T}}_{n})\Vert\rightarrow
\infty$ as $N\rightarrow\infty$. Then, we can similarly derive that the scale
of the $\widehat{\eta}(\mathbf{\check{T}}_{N})^{-1}m_{Ni}^{jl}(\widehat
{\mathbf{\beta}}(\mathbf{\check{T}}_{N}),\widehat{\mathbf{\gamma}%
}(\mathbf{\check{T}}_{N}))$ for $1\leq i\leq n,$ $1\leq j,l\leq p$ tends to 0
and this contradicts again the fact that it is constantly equal to $s_{0}$.
\end{proof}

\begin{proof}
[Proof of Theorem \ref{teo:classicbp}]The proof of Theorem \ref{teo:classicbp}
follows immediately from Lemmas \ref{lemma:bdpgamma}, \ref{lemma:bdpbeta} and
\ref{lemma:bdpeta}.
\end{proof}

\begin{proof}
[Proof of Theorem \ref{teo:independentbp}]The proof of Theorem
\ref{teo:independentbp} follows immediately from Lemmas \ref{lemma:bdpgamma},
\ref{lemma:bdpbeta} and \ref{lemma:bdpeta} once we notice that the results of
aforementioned Lemmas will continue to hold if the total number of
contaminated rows (in one or both columns) for each pair $(j,l)$ will be less
than $bn$. To ensure this fact under the independent contamination model, it
is sufficient to consider a contamination level not greater than $b/2$.
\end{proof}

\section{Asymptotic properties \label{app:asymptotic}}

Hereafter we prove the Fisher Consistency of the estimating functional
associated to the composite $\tau$-estimator. Let $(\mathbf{y,x})$ with
distribution $F$, then given $\mathbf{\beta}$ and $\mathbf{\gamma,}$ the
asymptotic M-scales $s_{jl}^{a}(\mathbf{\beta},\mathbf{\gamma},F)$ are defined
by
\begin{equation*}
E\left(  \rho_{1}\left(  \frac{m(\mathbf{y}^{jl},\mathbf{\mu}^{jl}%
(\mathbf{\beta}),\Sigma_{jl}^{\ast}(\mathbf{\gamma}))}{s_{jl}^{a}%
(\mathbf{\beta},\mathbf{\gamma},F)}\right)  \right)  =b,
\end{equation*}
and the asymptotic $\tau$-scales $\tau_{jl}^{a}(\mathbf{\beta},\mathbf{\gamma
},F)$ by
\begin{equation*}
\tau_{jl}^{a}(\mathbf{\beta},\mathbf{\gamma},F)=s_{jl}^{a}(\mathbf{\beta
},\mathbf{\gamma},F)E\left(  \rho_{2}\left(  \frac{m(\mathbf{y}^{jl}%
,\mathbf{\mu}^{jl}(\mathbf{\beta}),\Sigma_{jl}^{\ast}(\mathbf{\gamma}%
))}{s_{jl}^{a}(\mathbf{\beta},\mathbf{\gamma},F)}\right)  \right)  .
\end{equation*}
Finally, we define the asymptotic composite $\tau$ loss function as
\begin{equation*}
T^{a}(\mathbf{\beta},\mathbf{\gamma},F)=\sum_{j=1}^{p-1}\sum_{l=j+1}^{p}%
\tau_{jl}^{a}(\mathbf{\beta},\mathbf{\gamma},F).
\end{equation*}
Then the estimating functional $(\mathbf{B}(F),\mathbf{G}(F))$ of
$(\mathbf{\beta},\mathbf{\gamma})$ associated to the composite $\tau
$-estimators is defined by%
\begin{equation}
(\mathbf{B}(F),\mathbf{G}(F))=\arg\min_{(\mathbf{\beta,\gamma)}}%
T^{a}(\mathbf{\beta},\mathbf{\gamma},F), \label{esfuntau}%
\end{equation}
and the composite $\tau$-estimator of $(\mathbf{\beta},\mathbf{\gamma})$ can
be defined by
\begin{equation}
(\widehat{\mathbf{\beta}},\widehat{\mathbf{\gamma}})=(\mathbf{B}%
(F_{n}),\mathbf{G}(F_{n})), \label{fuctau}%
\end{equation}
where $F_{n}$ is the empirical distribution of $(\mathbf{y}_{1},\mathbf{x}%
_{1}),\ldots,(\mathbf{y}_{n},\mathbf{x}_{n})$.

Now we can state the theorem establishing the Fisher consistency of the
estimating functional associated to the compose $\tau$-estimators.

\begin{theorem}
\label{theorem:fisher} Let $(\mathbf{y},\mathbf{x})$ have distribution $F_{0}$
and call $H_{0}$ the marginal distribution of $\mathbf{x}$. Assume (i)
$\rho_{1}$ satisfies (A1-A5), (ii) $\rho_{2}$ satisfies A1-A6, (iii) under
$F_{0}$ A7 and A8 holds and (iv) A9. Then, if $(\mathbf{\beta},\mathbf{\gamma
})\neq(\mathbf{\beta}_{0},\mathbf{\gamma}_{0})$
\begin{equation*}
T^{a}(\mathbf{\beta},\mathbf{\gamma},F_{0})>T^{a}(\mathbf{\beta}%
_{0},\mathbf{\gamma}_{0},F_{0}),
\end{equation*}
that is, $(\mathbf{B}(F_{0}),\mathbf{G}(F_{0}))=(\mathbf{\beta}_{0}%
,\mathbf{\gamma}_{0}).$
\end{theorem}

The following lemmas are required to prove this Theorem.

\begin{lemma}
[\citet{garciaben_martinez_yohai_2006}, A.10]\label{lemma:one} Suppose that
$\rho$ satisfies A1-A5 and $\mathbf{u}$ is a random vector of dimension $h$
with density given by (\ref{denseli}) with $\Sigma=\Sigma_{0}$ and
$f_{0}^{\ast}$ non increasing and with at least one point of decrease in the
interval where $\rho$ is strictly increasing. Let $\mathbf{v}$ be a random
vector independent of $\mathbf{u}$ and $\Sigma$ a scatter matrix such that
$|\Sigma|=|\Sigma_{0}|.$ Then%
\begin{equation}
E\left(  \rho\left(  (\mathbf{u}-\mathbf{v})^{\top}\Sigma^{-1}(\mathbf{u}%
-\mathbf{v})\right)  \right)  \geq E\left(  \rho\left(  \mathbf{u}^{\top
}\Sigma_{0}^{-1}\mathbf{u}\right)  \right)  . \label{lem10-1}%
\end{equation}
Moreover, if either (i) $P(\mathbf{v}\neq\mathbf{0})>0$ or (ii) $\Sigma
\neq\Sigma_{0}$, then
\begin{equation}
E\left(  \rho\left(  (\mathbf{u}-\mathbf{v})^{\top}\Sigma^{-1}(\mathbf{u}%
-\mathbf{v})\right)  \right)  >E\left(  \rho\left(  \mathbf{u}^{\top}%
\Sigma_{0}^{-1}\mathbf{u}\right)  \right)  . \label{lem10-2}%
\end{equation}

\end{lemma}

Using the above result we can prove the following this Lemma.

\begin{lemma}
\label{lemma:two} Suppose that $\rho$ satisfies A1-A5 and let $\mathbf{u}$ and
$\mathbf{v}$ be as in Lemma \ref{lemma:one} and let $\Sigma$ be a $h\times h$
positive definite symmetric matrix. Put $\Sigma_{0}^{\ast}=\Sigma_{0}%
/|\Sigma_{0}|^{1/h}$ and $\Sigma^{\ast}=\Sigma/|\Sigma|^{1/h}$, then
\begin{equation}
E\left(  \rho\left(  (\mathbf{u}-\mathbf{v})^{\top}\Sigma^{\ast-1}%
(\mathbf{u}-\mathbf{v})\right)  \right)  \geq E\left(  \rho\left(
\mathbf{u}^{\top}\Sigma_{0}^{\ast-1}\mathbf{u}\right)  \right)  .
\label{quaop}%
\end{equation}
Moreover, suppose that either (i) $P(\mathbf{v}\neq\mathbf{0})>0$ or (ii)
$\Sigma\neq\alpha\Sigma_{0}$ for some $\alpha>0$ then
\begin{equation}
E\left(  \rho\left(  (\mathbf{u}-\mathbf{v})^{\top}\Sigma^{\ast-1}%
(\mathbf{u}-\mathbf{v})\right)  \right)  >E\left(  \rho\left(  \mathbf{u}%
^{\top}\Sigma_{0}^{\ast-1}\mathbf{u}\right)  \right)  . \label{lem2fin}%
\end{equation}

\end{lemma}

\begin{proof}
Let $\tilde{\rho}(u)=\rho(u|\Sigma_{0}|^{1/h})$. Clearly $\tilde{\rho}$
satisfies A1-A5 too. Put $\Sigma_{1}=\Sigma|\Sigma|^{-1/h}|\Sigma_{0}|^{1/h}$
and note that $\Sigma_{1}$ is different from $\Sigma_{0}$ but with the same
determinant. We have that%
\begin{align*}
E\left(  \rho\left(  \mathbf{u}^{\top}\Sigma_{0}^{\ast-1}\mathbf{u}\right)
\right)   &  =E\left(  \rho\left(  \mathbf{u}^{\top}\Sigma_{0}^{-1}|\Sigma
_{0}|^{1/h}\mathbf{u}\right)  \right) \\
&  =E\left(  \tilde{\rho}\left(  \mathbf{u}^{\top}\Sigma_{0}^{-1}%
\mathbf{u}\right)  \right)
\end{align*}
and%
\begin{equation*}
E\left(  \rho\left(  (\mathbf{u}-\mathbf{v})^{\top}\Sigma^{\ast
-1}(\mathbf{u}-\mathbf{v})\right)  \right)  =E\left(  \tilde{\rho}\left(
(\mathbf{u}-\mathbf{v})^{\top}\Sigma_{1}^{-1}(\mathbf{u}-\mathbf{v})\right)
\right)  .
\end{equation*}
Then Lemma \ref{lemma:two} follows from Lemma \ref{lemma:one}
\end{proof}

\begin{lemma}
\label{lemma:three} Assume that (i) $\rho_{1}$ satisfies A1-A5, (ii) $\rho
_{2}$ satisfies A1, A6 (iii) A7 holds with $\mathbf{\beta}=\mathbf{\beta}_{0}$
and $\mathbf{\gamma}=\mathbf{\gamma}_{0}$. Then, $s_{jl}^{a}(\mathbf{\beta
},\mathbf{\gamma})\geq s_{jl}^{a}(\mathbf{\beta}_{0},\mathbf{\gamma}_{0})$ and
$\tau_{jl}^{a}(\mathbf{\beta},\mathbf{\gamma})\geq\tau_{jl}^{a}(\mathbf{\beta
}_{0},\mathbf{\gamma}_{0})$ for all couples $(j,l)$ Moreover, if for the pairs
$(j,l)$ either
\begin{equation}
P(\mathbf{x}^{jl}\mathbf{\beta}-\mathbf{x}^{jl}\mathbf{\beta}_{0}%
\neq\mathbf{0})>0\quad\text{or}\quad\Sigma_{jl}(1,\mathbf{\gamma})\neq
\alpha\Sigma_{jl}(1,\mathbf{\gamma}_{0}) \label{assor}%
\end{equation}
for all $\alpha>0,$ then $s_{jl}^{a}(\mathbf{\beta},\mathbf{\gamma}%
)>s_{jl}^{a}(\mathbf{\beta}_{0},\mathbf{\gamma}_{0})$ and $\tau_{jl}%
^{a}(\mathbf{\beta},\mathbf{\gamma})>\tau_{jl}^{a}(\mathbf{\beta}%
_{0},\mathbf{\gamma}_{0})$.
\end{lemma}

\begin{proof}
Let $\mathbf{q}=\mathbf{x}(\mathbf{\beta}-\mathbf{\beta}_{0})$, then
$\mathbf{y}-\mathbf{x}\mathbf{\beta}=\mathbf{u}-\mathbf{q}$, where
$\mathbf{u}=\mathbf{y}-\mathbf{x}\mathbf{\beta}$. Observe that $\mathbf{q}$
depends only on $\mathbf{x}$ and hence it is independent of $\mathbf{u}$.
Moreover $\mathbf{u}^{jl}$ has an elliptical distribution with density of the
form%
\begin{equation*}
\frac{g(\mathbf{u}^{jl\top}\Sigma_{jl}(\eta_{0},\mathbf{\gamma}_{0}%
)^{-1}\mathbf{u}^{jl})}{\left\vert \Sigma_{jl}(\eta_{0},\mathbf{\gamma}%
_{0})\right\vert ^{1/2}},
\end{equation*}
where $g$ is non increasing and strictly increasing in a neighborhood of $0$.
Them, by Lemma \ref{lemma:two} we have%
\begin{align}
&  E\left(  \rho_{1}\left(  \frac{m(\mathbf{y}^{jl},\mathbf{\mu}%
^{jl}(\mathbf{\beta}),\Sigma_{jl}^{\ast}(\mathbf{\gamma}))}{s_{jl}%
^{a}(\mathbf{\beta}_{0},\mathbf{\gamma}_{0})}\right)  \right)  =\nonumber\\
&  =E\left(  \rho_{1}\left(  \frac{\left(  \mathbf{x}^{jl}(\mathbf{\beta}%
_{0}-\mathbf{\beta})+\mathbf{u}^{jl}\right)  ^{\top}\Sigma_{jl}^{\ast
-1}(\mathbf{\gamma})\left(  \mathbf{x}^{jl}(\mathbf{\beta}_{0}-\mathbf{\beta
})+\mathbf{u}^{jl}\right)  }{s_{jl}^{a}(\mathbf{\beta}_{0},\mathbf{\gamma}%
_{0})}\right)  \right) \nonumber\\
&  =E\left(  \rho_{1}\left(  \frac{\left(  \mathbf{u}^{jl}-\mathbf{q}%
^{jl}\right)  ^{\top}\Sigma_{jl}^{\ast-1}(\mathbf{\gamma})\left(
\mathbf{u}^{jl}-\mathbf{q}^{jl}\right)  }{s_{jl}^{a}(\mathbf{\beta}%
_{0},\mathbf{\gamma}_{0})}\right)  \right) \nonumber\\
&  \geq E\left(  \rho_{1}\left(  \frac{\mathbf{u}^{jl\top}\Sigma_{jl}^{\ast
-1}(\mathbf{\gamma}_{0})\mathbf{u}^{jl}}{s_{jl}^{a}(\mathbf{\beta}%
_{0},\mathbf{\gamma}_{0})}\right)  \right) \label{equ:scale}\\
&  =E\left(  \rho_{1}\left(  \frac{m(\mathbf{y}^{jl},\mathbf{\mu}%
^{jl}(\mathbf{\beta}_{0}),\Sigma_{jl}^{\ast}(\mathbf{\gamma}_{0}))}{s_{jl}%
^{a}(\mathbf{\beta}_{0},\mathbf{\gamma}_{0})}\right)  \right)  =b,\nonumber
\end{align}
and therefore $s_{jl}^{a}(\mathbf{\beta},\gamma)\geq s_{jl}^{a}(\mathbf{\beta
}_{0},\mathbf{\gamma}\mathbf{_{0}})$. Under at least one of the two
assumptions in (\ref{assor}), by Lemma \ref{lemma:two} the inequality in
(\ref{equ:scale}) becomes a strict inequality and hence $s_{jl}^{a}%
(\mathbf{\beta},\mathbf{\gamma})>s_{jl}^{a}(\mathbf{\beta}_{0},\mathbf{\gamma
}_{0})$.

Lemma A.8 in \citet{garciaben_martinez_yohai_2006} proves that the function
$\tau(s)=s\ E(\rho_{2}(v/s))$ is a non-decreasing function of $s$ for any non
negative value $v$ under A1-A6. Using this result and since $s_{jl}%
^{a}(\mathbf{\beta},\mathbf{\gamma})>s_{jl}^{a}(\mathbf{\beta}_{0}%
,\mathbf{\gamma}_{0})$ we have
\begin{align*}
\tau_{jl}^{2}(\mathbf{\beta},\mathbf{\gamma})  &  =s_{jl}^{a}(\mathbf{\beta
},\mathbf{\gamma})E\left(  \rho_{2}\left(  \frac{m(\mathbf{y}^{jl}%
,\mathbf{\mu}^{jl}(\mathbf{\beta}),\Sigma_{jl}^{\ast}(\mathbf{\gamma}%
))}{s_{jl}^{a}(\mathbf{\beta},\mathbf{\gamma})}\right)  \right) \\
&  \geq s_{jl}^{a}(\mathbf{\beta}_{0},\mathbf{\gamma}_{0})E\left(  \rho
_{2}\left(  \frac{m(\mathbf{y}^{jl},\mathbf{\mu}^{jl}(\mathbf{\beta}%
),\Sigma_{jl}^{\ast}(\mathbf{\gamma}))}{s_{jl}^{a}(\mathbf{\beta}%
_{0},\mathbf{\gamma}_{0})}\right)  \right) \\
&  \geq s_{jl}^{a}(\mathbf{\beta}_{0},\mathbf{\gamma}_{0})E\left(  \rho
_{2}\left(  \frac{m(\mathbf{y}^{jl},\mathbf{\mu}^{jl}(\mathbf{\beta}%
_{0}),\Sigma_{jl}^{\ast}(\mathbf{\gamma}_{0}))}{s_{jl}^{a}(\mathbf{\beta}%
_{0},\mathbf{\gamma}_{0})}\right)  \right)  =\tau_{jl}^{2}(\mathbf{\beta}%
_{0},\mathbf{\gamma}_{0}).
\end{align*}
When one of the two assumptions in (\ref{assor}) hold, the last inequality is
strict proving that $\tau_{jl}^{2}(\mathbf{\beta},\mathbf{\gamma})>\tau
_{jl}^{2}(\mathbf{\beta}_{0},\mathbf{\gamma}_{0})$.
\end{proof}

\begin{proof}
[Proof of Theorem \ref{theorem:fisher}]

Consider first the case of $\mathbf{\gamma}\neq\mathbf{\gamma}_{0}$. According
to the Lemma \ref{lemma:three} it is enough to show that $\Sigma
(1,\mathbf{\gamma})\neq\Sigma(1,\mathbf{\gamma}_{0})$ implies that there
exists at least one pair $(j,l)$ such that $\Sigma_{jl}^{\ast}(\mathbf{\gamma
})\neq\Sigma_{jl}^{\ast}(\mathbf{\gamma}_{0})$. We prove it by contradiction.
Let us assume that $\Sigma_{jl}^{\ast}(\mathbf{\gamma})=\Sigma_{jl}^{\ast
}(\mathbf{\gamma}_{0})$ for all $1\leq j\leq l\leq p$ this implies that for
all $(j,l)$ there exists $\alpha_{jl}$ such that $\Sigma_{jl}(1,\mathbf{\gamma
})=\alpha_{jl}\Sigma_{jl}(1,\mathbf{\gamma}_{0})$. However $\alpha_{jl}%
=\alpha_{j^{\prime}l}$ since the corresponding matrices have one common
element and similarly we can prove that $\alpha_{jl}=\alpha_{jl^{\prime}}$.
Then all $\alpha_{jl}$ are equals to a same value $\alpha$. Then
$\Sigma(1,\mathbf{\gamma})=\alpha\Sigma(1,\mathbf{\gamma}_{0})$ contradicting
A9. Consider now the case $\mathbf{\delta}=\mathbf{\beta}-\mathbf{\beta}%
_{0}\neq\mathbf{0}$. It will be enough, according to Lemma \ref{lemma:three},
that there exists a pair $(j,l)$ such that $P(\mathbf{x}^{jl}\mathbf{\delta
}\neq\mathbf{0})>0$. For this to be true it is enough to show that there
exists a $j$ such that $P(\mathbf{x}^{j}\mathbf{\delta}\neq0)>0$ where
$\mathbf{x}^{j}$ stands for the $j$ row of the $p\times k$ matrix $\mathbf{x}%
$. Let us assume that $P(\mathbf{x}^{j}\mathbf{\delta}\neq0)=0$ for all $1\leq
j\leq p$. Since $\{\mathbf{x}\mathbf{\delta}\neq\mathbf{0}\}=\cup_{j=1}%
^{p}\{\mathbf{x}^{j}\mathbf{\delta}\neq0\}$ then
\begin{equation*}
0\leq P(\mathbf{x}\mathbf{\delta}\neq\mathbf{0})=P\left(  \cup_{j=1}%
^{p}\{\mathbf{x}^{j}\mathbf{\delta}\neq0\}\right)  \leq\sum_{j=1}%
^{p}P(\mathbf{x}^{j}\mathbf{\delta}\neq0)=0.
\end{equation*}
And this contradict the assumption A8.
\end{proof}

\begin{proof}
[Heuristic proof of Theorem \ref{teo:consistency}]It can be proved that the
functional $(\mathbf{B}(F),\mathbf{G}(F))$ defined in (\ref{esfuntau}) is
continuous at $F_{0}$ with the topology associated to the convergence in
distribution. Let $F_{n}$ be the empirical distribution of $(\mathbf{y}%
_{i},\mathbf{x}_{i})$, $1\leq i\leq n$. Then $F_{n}\overset{d}{\rightarrow
}F_{0}$ a.s. ( where $\overset{d}{\rightarrow}$ denotes weak convergence).
Then $(\widehat{\mathbf{\beta}},\widehat{\mathbf{\gamma}})=(\mathbf{B}%
(F_{n}),\mathbf{G}(F_{n}))\overset{d}{\rightarrow}(\mathbf{\beta}%
_{0},\mathbf{\gamma}_{0})$.

Note that
\begin{equation*}
(\mathbf{y}_{i}^{jl}-\mathbf{x}_{i}^{jl}\widehat{\mathbf{\beta}})^{\top}%
\Sigma_{jl}(1,\widehat{\mathbf{\gamma}})^{-1}(\mathbf{y}_{i}^{jl}%
-\mathbf{x}_{i}^{jl}\widehat{\mathbf{\beta}})
\end{equation*}
has approximately the distribution of $\eta v$, where $v$ has chi square
distribution with two degree of freedom. Therefore by (\ref{Msc2}) given any
$\varepsilon>0$ we have%
\begin{align}
\frac{2}{p(p-1)n}\sum_{i=1}^{n}\sum_{j=1}^{p-1}\sum_{l=j+1}^{p}\rho\left(
\frac{(\mathbf{y}_{i}^{jl}-\mathbf{x}_{i}^{jl}\widehat{\mathbf{\beta}})^{\top
}\Sigma_{jl}(1,\widehat{\mathbf{\gamma}})^{-1}(\mathbf{y}_{i}^{jl}%
-\mathbf{x}_{i}^{jl}\widehat{\mathbf{\beta}})}{\eta_{0}+\varepsilon}\right)
&  \overset{\text{a.s.}}{\rightarrow}E\left(  \rho\left(  \frac{\eta_{0}%
v}{\eta_{0}+\varepsilon}\right)  \right) \nonumber\\
&  < b. \label{est1}%
\end{align}
Similarly%
\begin{align}
\frac{2}{p(p-1)n}\sum_{i=1}^{n}\sum_{j=1}^{p-1}\sum_{l=j+1}^{p}\rho\left(
\frac{(\mathbf{y}_{i}^{jl}-\mathbf{x}_{i}^{jl}\widehat{\mathbf{\beta}})^{\top
}\Sigma_{jl}(1,\widehat{\mathbf{\gamma}})^{-1}(\mathbf{y}_{i}^{jl}%
-\mathbf{x}_{i}^{jl}\widehat{\mathbf{\beta}})}{\eta_{0}-\varepsilon}\right)
&  \overset{\text{a.s.}}{\rightarrow}E\left(  \rho\left(  \frac{\eta_{0}%
v}{\eta_{0}-\varepsilon}\right)  \right) \nonumber\\
&  >b. \label{est2}%
\end{align}
Therefore by (\ref{equ:eta0}) with probability one there exist $n_{0}$ such
that for $n\geq n_{0},\eta_{0}-\varepsilon<\widehat{\eta}<\eta_{0}%
+\varepsilon.$ This implies that $\widehat{\eta}\rightarrow\eta_{0}$ a.s..
\end{proof}

\bibliography{pwlmm}
\end{document}